\documentclass{amsart}
\usepackage{amscd}
\usepackage{amssymb}
\usepackage{enumerate}
\usepackage[T1]{fontenc}
\usepackage{hyperref}
\usepackage{mathrsfs}
\usepackage{stackrel}
\usepackage[all]{xy}
\usepackage{yhmath}

\DeclareMathAlphabet{\mathpzc}{OT1}{pzc}{m}{it}
\DeclareMathOperator{\gr}{gr}
\DeclareMathOperator{\Hom}{Hom}
\DeclareMathOperator{\End}{End}

\DeclareMathOperator{\Der}{Der}

\DeclareMathOperator{\Sp}{Sp}

\DeclareMathOperator{\Sym}{Sym}

\DeclareMathOperator{\rig}{rig}
\DeclareMathOperator{\rk}{rk}

\DeclareMathOperator{\Coh}{Coh}
\DeclareMathOperator{\Loc}{Loc}

\DeclareMathOperator{\Dp}{dp}
\DeclareMathOperator{\id}{id}

\DeclareMathOperator{\Ab}{Ab}
\DeclareMathOperator{\Supp}{Supp}
\DeclareMathOperator{\op}{op}

\begin{document}
\newtheorem{MainThm}{Theorem}
\renewcommand{\theMainThm}{\Alph{MainThm}}
\theoremstyle{definition}
\newtheorem*{examps}{Examples}
\newtheorem*{defn}{Definition}
\newtheorem*{notn}{Notation}
\theoremstyle{remark}
\newtheorem*{rmk}{Remark}
\newtheorem*{rmks}{Remarks}
\theoremstyle{plain}
\newtheorem*{lem}{Lemma}
\newtheorem*{prop}{Proposition}
\newtheorem*{thm}{Theorem}
\newtheorem*{example}{Example}
\newtheorem*{examples}{Examples}
\newtheorem*{cor}{Corollary}
\newtheorem*{conj}{Conjecture}
\newtheorem*{hyp}{Hypothesis}
\newtheorem*{thrm}{Theorem}
\newtheorem*{quest}{Question}
\theoremstyle{remark}

\newcommand{\Zp}{{\mathbb{Z}_p}}
\newcommand{\Qp}{{\mathbb{Q}_p}}
\newcommand{\Fp}{{\mathbb{F}_p}}
\newcommand{\A}{\mathcal{A}}
\newcommand{\B}{\mathcal{B}}
\newcommand{\C}{\mathcal{C}}
\newcommand{\D}{\mathcal{D}}
\newcommand{\E}{\mathcal{E}}
\newcommand{\F}{\mathcal{F}}
\newcommand{\G}{\mathcal{G}}
\newcommand{\I}{\mathcal{I}}
\newcommand{\J}{\mathcal{J}}
\renewcommand{\L}{\mathcal{L}}
\newcommand{\sL}{\mathscr{L}}
\newcommand{\M}{\mathcal{M}}
\newcommand{\sM}{\mathscr{M}}
\newcommand{\N}{\mathcal{N}}
\newcommand{\sN}{\mathscr{N}}
\renewcommand{\O}{\mathcal{O}}
\newcommand{\cP}{\mathcal{P}}
\newcommand{\R}{\mathcal{R}}
\newcommand{\cS}{\mathcal{S}}
\newcommand{\T}{\mathcal{T}}
\newcommand{\U}{\mathcal{U}}
\newcommand{\sU}{\mathscr{U}}
\newcommand{\sV}{\mathscr{V}}
\newcommand{\V}{\mathcal{V}}
\newcommand{\W}{\mathcal{W}}
\newcommand{\sW}{\mathscr{W}}
\newcommand{\X}{\mathcal{X}}
\newcommand{\Y}{\mathcal{Y}}
\newcommand{\Z}{\mathcal{Z}}
\newcommand{\PFS}{\mathbf{PreFS}}
\newcommand{\h}[1]{\widehat{#1}}
\newcommand{\hA}{\h{A}}
\newcommand{\hK}[1]{\h{#1_K}}
\newcommand{\hsULK}{\hK{\sU(\L)}}
\newcommand{\hsUFK}{\hK{\sU(\F)}}
\newcommand{\hsULnK}{\hK{\sU(\pi^n\L)}}
\newcommand{\hn}[1]{\h{#1_n}}
\newcommand{\hnK}[1]{\h{#1_{n,K}}}
\newcommand{\w}[1]{\wideparen{#1}}
\newcommand{\wK}[1]{\wideparen{#1_K}}
\newcommand{\invlim}{\varprojlim}
\newcommand{\dirlim}{\varinjlim}
\newcommand{\fr}[1]{\mathfrak{{#1}}}
\newcommand{\LRU}[1]{\sU(\mathscr{#1})}
\newcommand{\et}{\acute et}

\newcommand{\ts}[1]{\texorpdfstring{$#1$}{}}
\newcommand{\st}{\mid}
\newcommand{\be}{\begin{enumerate}[{(}a{)}]}
\newcommand{\ee}{\end{enumerate}}
\newcommand{\qmb}[1]{\quad\mbox{#1}\quad}
\let\le=\leqslant  \let\leq=\leqslant
\let\ge=\geqslant  \let\geq=\geqslant

\title{$\w{\D}$-modules on rigid analytic spaces II: Kashiwara's equivalence}
\author{Konstantin Ardakov}
\address{Mathematical Institute\\University of Oxford\\Oxford OX2 6GG}
\email{ardakov@maths.ox.ac.uk}

\author{Simon Wadsley}
\address{Homerton College, Cambridge, CB2 8PH}
\email{S.J.Wadsley@dpmms.cam.ac.uk}

\thanks{The first author was supported by EPSRC grant EP/L005190/1. \\ \hspace*{0.33cm} To appear in the Journal of Algebraic Geometry.}
\begin{abstract} Let $X$ be a smooth rigid analytic space. We prove that the category of coadmissible $\w\D_X$-modules supported on a closed smooth subvariety $Y$ of $X$ is naturally equivalent to the category of coadmissible $\w\D_Y$-modules, and use this result to construct a large family of pairwise non-isomorphic simple coadmissible $\w\D_X$-modules.
\end{abstract}
\maketitle
\setcounter{tocdepth}{1}  
\tableofcontents
\section{Introduction}
\subsection{Modules over rings of infinite-order differential operators}
In our earlier work \cite{DCapOne} we introduced a sheaf of infinite-order differential operators $\w{\D}$ on every smooth rigid analytic space $X$. We regard $\w\D$ as a ``rigid analytic quantisation'' of the cotangent bundle of $X$: by construction, the set of sections of $\w\D$ over any sufficiently small admissible open subset of $X$ is naturally in bijection with the algebra of rigid analytic functions on the entire cotangent bundle of that open subset.

There are two main motivations for studying $\w\D$-modules. The sheaf $\D^\infty$ of infinite-order differential operators on complex manifolds was introduced by Sato's school in the early 70s as part of their microlocal approach to the theory linear partial differential equations  \cite{SKK}. It plays an important role in the classical Riemann-Hilbert correspondence \cite{KashRH1984}, \cite{MebkhRH}: regular holonomic $\D$-modules are derived equivalent not only to constructible sheaves, but also to holonomic $\D^\infty$-modules \cite{KK3}, \cite{MebkhBiDual}. The reader can find a modern treatment of the sheaf $\D^\infty$ as well as its role in the Riemann-Hilbert correspondence in Chapters III and VIII of Bjork's book \cite{BjorkAnDmod}. It has recently become apparent that our sheaf $\w\D$ in the $p$-adic setting is analogous in more than one way to the complex analytic $\D^\infty$: for example, in a forthcoming publication \cite{ABB} it will be shown that if the ground field $K$ is sufficiently large then $\w\D$ is naturally isomorphic to the sheaf of continuous $K$-linear endomorphisms of the structure sheaf $\O$. This $p$-adic analogue of a result of Ishimura \cite{Ishimura} encourages us to hope that some $p$-adic analogue of the Riemann-Hilbert correspondence can be established for an appropriate derived category of $\w\D$-modules. A more modest goal would be to establish a version of the Biduality/Reconstruction Theorem of \cite{KK3} and \cite{MebkhBiDual}, perhaps following the approach of \cite{ProsSch} and \cite{BBK}.

However, our main motivation for introducing and studying $\w\D$ on rigid analytic spaces is the belief that a detailed understanding of coadmissible $\w\D$-modules will lead to significant progress in the theory of \emph{locally analytic representations} of $p$-adic analytic groups. The category of admissible locally analytic representations of a $p$-adic analytic group $G$ is by definition anti-equivalent to the category of coadmissible modules over the locally analytic distribution algebra $D(G,K)$ of $G$. This distribution algebra naturally contains the Arens--Michael envelope $\w{U(\fr{g})}$ of the universal enveloping algebra of the Lie algebra $\mathfrak{g}$ of $G$ as the algebra of locally analytic distributions on $G$ supported at the identity. In \cite{DCapOne}, we introduced the category of \emph{coadmissible $\w{\D}$-modules} on $X$ and showed that it shares many of the features of the classical category of coherent $\D$-modules on a complex analytic manifold. In future work, we will establish an analogue of the Beilinson-Bernstein Localisation Theorem for coadmissible $\w{U(\fr{g})}$-modules by relating them to coadmissible $\w\D$-modules on an appropriate rigid analytic flag variety; the main results of the present paper, for example, Theorem C below, will then be used to give a \emph{geometric construction} of a large class of new examples of irreducible coadmissible $\w{U(\fr{g})}$-modules. 

\subsection{Main results}
An early fundamental result in the classical theory of $\D$-modules due to Kashiwara says that if $Y\to X$ is a closed embedding of smooth algebraic varieties then there is a natural equivalence of categories between the category of (coherent) $\D$-modules on $X$ that are supported on $Y$ and the category of (coherent) $\D$-modules on $Y$. In this paper we prove a version of this as follows.

\begin{MainThm}\label{Kashiwara}  Let $X$ be a smooth rigid analytic variety over a complete discretely valued field $K$ of characteristic zero. Let $Y$ be a smooth closed analytic subset of $X$. There is a natural equivalence of categories 
\[\left\{ 
				\begin{array}{c} 
					coadmissible\\
					\w\D_Y-\hspace{-0.1cm}modules\hspace{0.1cm}
				\end{array}
\right\} \cong \left\{
				\begin{array}{c}
				 coadmissible\hspace{0.1cm}\w\D_X\hspace{-0.1cm}-\hspace{-0.1cm}modules\\
				 supported \hspace{0.1cm} on \hspace{0.1cm} Y
				\end{array}
\right\}.\]
\end{MainThm}


We should note in passing that there is a theory of arithemetic differential operators on formal schemes developed by Berthelot and that is being used by Huyghe, Patel, Schmidt and Strauch (see \cite{HPSS} for example) as alternative to our approach to the study of locally analytic representations of $p$-adic analytic groups. In \cite[\S5.3.1]{Berth2} Berthelot observes that, for a fixed finite level $m$, the analogue of Theorem \ref{Kashiwara} is false for his sheaves of rings $\hat{\mathscr{D}}_{\mathbb{Q}}^{[m]}$ and closed embeddings of smooth formal schemes. In \cite[Theorem 5.3.3]{Berth2} he states a version for his sheaves of rings $\hat{\mathscr{D}}_{\mathbb{Q}}^\dag$; the proof of this result is yet to appear.

Our next result forms a link between coadmissible $\w\D_X$-modules and the more classical theory of $p$-adic differential equations.
\begin{MainThm}\label{MainOxCoh}Let $X$ be a smooth rigid analytic space. Then the forgetful functor 
\[\left\{ 
				\begin{array}{c} 
					\O_X-coherent \hspace{0.1cm} \\
					coadmissible\hspace{0.1cm} \w\D_X-\hspace{-0.1cm}modules\hspace{0.1cm}
				\end{array}
\right\} \longrightarrow \left\{
				\begin{array}{c}
				 \O_X-coherent \hspace{0.1cm} \\
				 \D_X \hspace{-0.1cm}-\hspace{-0.1cm}modules\hspace{0.1cm} 
				\end{array}
\right\}\]
is an equivalence of categories.
\end{MainThm}

Using Theorems \ref{Kashiwara} and \ref{MainOxCoh}, we can construct a large family of pairwise non-isomorphic simple objects in the category of coadmissible $\w\D_X$-modules.

\begin{MainThm}\label{SimpleDcaps}
Let $X$ be a smooth rigid analytic variety over a complete discretely valued field $K$ of characteristic zero.
\be
\item $\iota_+ \O_Y$ is a simple coadmissible $\w\D_X$-module whenever $\iota : Y \to X$ is the inclusion of a smooth, connected, closed subvariety $Y$. 
\item If $\iota' : Y\to X$ is another such inclusion and $\iota_+ \O_Y \cong \iota'_+ \O_{Y'}$ as coadmissible $\w\D_X$-modules, then $Y = Y'$.
\ee 
\end{MainThm}

\subsection{Structure of this paper}
As in our previous paper \cite{DCapOne} we work in a more general framework than that described above: for each Lie algebroid $\sL$ on a rigid analytic space $X$ we have a sheaf of rings $\w{\sU(\sL)}$ that we call the completed enveloping algebra. When $X$ is smooth and $\sL=\T_X$, $\w{\sU(\sL)}=\w\D_X$. Our main results in the general framework depend on the normal sheaf $\mathcal{N}_{Y/X}$ being locally free on $Y$ and on the surjectivity of the natural map $\iota^\ast\sL\to \mathcal{N}_{Y/X}$. When $X$ and $Y$ are both smooth and $\sL=\T_X$ these conditions are automatically satisfied. As we outline the structure of the paper we will suppress mention of these and similar technical global conditions on these sheaves that we need to impose for intermediate steps. Full details may be found in the body of the paper.

In section \ref{Review} we review material from \cite{DCapOne}, restating results that we will use in this paper, for the convenience of the reader. 

In section \ref{LRSwitch} we establish an equivalence of categories between the category of coadmissible left modules and the category of coadmissible right modules for any completed enveloping algebra. This equivalence is entirely analogous to the classical equivalence for $\D$-modules and depends on the existence of a line bundle $\Omega_{\sL}$ on $X$ that may be equipped with a right action of the sheaf of completed enveloping algebras. This line bundle plays the role of $\Omega_X$ in the classical setting. 

Section \ref{ConthKUL} may be viewed as the heart of the paper with Theorem \ref{AffKash} providing a purely algebraic (right-module) version of Theorem \ref{Kashiwara} for Banach completions of $\w{\sU(\sL)}(X)$ when $X=\Sp(A)$ is affinoid and $Y=\Sp(A/(F))$. Most of the rest of the paper is concerned with proving that this result localises correctly. 

In section \ref{IplusInat} we consider the case that $Y\to X$ is a closed embedding of affinoid varieties. First we define the pullback $L_Y$ of a $(K,\O(X))$-Lie algebra $L$ to a $(K,\O(Y))$-Lie algebra. When $X$ and $Y$ are both smooth and $L=\T(X)$ this pullback is $\T(Y)$. After this we construct pushforward and pullback functors $i_+$ and $i^\natural$ between coadmissible right $\w{U(L)}$-modules and coadmissible right $\w{U(L_Y)}$-modules. These definitions follow the classical $\D$-module construction of pushforward and pullback via a `transfer bimodule' but some care is needed to ensure that $i^\natural$ preserves coadmissibility --- in fact some care is also needed for $i_+$ but the key results for that were established in \cite{DCapOne} and so that is less transparent in this paper. Section \ref{IplusInat} concludes with Theorem \ref{KashAffinoids} which is a (right-module) version of Theorem \ref{Kashiwara} for affinoids that is global on the base but `local along fibres' of the vector bundle underlying $L$. 

In Section \ref{Pullpush} we verify that the results and constructions of the previous sections localise correctly on the base, culminating in Theorem \ref{KashiwaraGeneral} that provides a right-module version of Theorem \ref{Kashiwara} in our fuller generality. 

Finally Section \ref{Main} brings all the above together to establish versions of Theorems \ref{Kashiwara}-\ref{SimpleDcaps} for general Lie algebroids. We also give some examples of Lie algebroids on $X$ other than $\T_X$ to which our results might be applied. The example of Atiyah algebras is especially pertinent for our representation theoretic applications since for these we will want to consider twisted versions of $\w{\D}$ that will arise as quotients of completed enveloping algebras of Atiyah algebras. 
  
It may be helpful at this point to trace different ways of interpreting the condition that a coadmissible $\w{\sU(\sL)}$ module $\sM$ is supported on $Y$ that appear at various points of the paper. 

In Theorem \ref{SuppMinfty} it is established that when $X$ is affinoid, a coadmissible $\w{\sU(\sL)}$-module $\sM$ is supported on $Y$ precisely if every element of $\O(X)$ that vanishes on $Y$ acts locally topologically nilpotently on $\sM(X)$. This result is a key step in the proof of Theorem \ref{Kashiwara} and is analogous to the familiar statement in algebraic geometry that, for $X$ affine, a quasi-coherent $\O_X$-module $\sM$ is supported on $V(I)$ precisely if each element of $I$ acts locally nilpotently on $\sM(X)$. 

In Corollary \ref{MinftyUsub} we are able to show that for every coadmissible $\w{\sU(\sL)}(X)$-module $M$ there is a natural coadmissible submodule $M_\infty(I)$ that consists of `sections supported on $Y$'; that is, those elements $m \in M$ with the property that every element $f$ of $\O(X)$ that vanishes on $Y$ satisfies $\lim_{n\to \infty}mf^n=0$. 

Proposition \ref{MinftyUsub} provides slightly more precise information about $M_{\infty}(I)$ in that it decribes the various Banach completions of $M_\infty(I)$ that realise it as a coadmissible module. This information is phrased in terms that enable us to apply Theorem \ref{AffKash} to these completions. In particular we see that the condition $M=M_{\infty}(I)$ corresponds precisely to each of these Banach completions $\h{M}$ satisfying the condition $\h{M}=\h{M}_{\Dp}(F)$ found in the statement of that Theorem (for a suitable choice of $F$ depending on the completion). This last condition is precisely what we need to construct the 'smoothing'  maps $e_j$ in \S\ref{EeJay} in order to show that the module $\h{M}$ is generated by its elements that are annihilated by $F$. 

\subsection{Acknowledgements} The authors would like to thank Kobi Kremnizer for pointing out the analogy between $\w\D$ and $\D^\infty$. They would also like to thank Oren Ben-Bassat and Joseph Bernstein for their interest in this work.

\subsection{Conventions} Throughout the paper $K$ will denote a complete discrete valuation field of characteristic zero with valuation ring  $\R$. We fix a non-zero non-unit $\pi \in \R$. Throughout Sections \ref{ConthKUL}, \ref{IplusInat} and \ref{Pullpush} we work with \emph{right} modules unless explicitly stated otherwise.

\section{Review of the basic theory of  \ts{\w\D}-modules}\label{Review}

In this section we review some of the notation and other material from \cite{DCapOne} that we will use in the remainder of the paper.

\subsection{Lie--Rinehart algebras and lattices}\label{RevLR}
Suppose that $R$ is a commutative ring and that $A$ is a commutative $R$-algebra. Recall that an \emph{$(R,A)$-Lie algebra} is a pair $(L,\rho)$ where \begin{itemize} \item $L$ is an $R$-Lie algebra and an $A$-module, and \item $\rho\colon L\to \Der_R(A)$ is an $A$-linear $R$-Lie algebra homomorphism \end{itemize} such that $[x,ay]=a[x,y]+\rho(x)(a)y$ for all $x,y\in L$ and $a\in A$. 

An $(R,A)$-Lie algebra $L$ is said to be \emph{coherent} if $L$ is coherent as an $A$-module and $L$ is said to be \emph{smooth} if it is both coherent and projective as an $A$-module.

Recall that a \emph{Tate algebra} is the subalgebra $K\langle x_1,\ldots,x_n \rangle$ of $K[[x_1,\ldots,x_n]]$ consisting of those power series $\sum_{\alpha\in\mathbb{N}^n} \lambda_\alpha x^\alpha$ with the property that $\lim_{\alpha \to \infty} \lambda_\alpha = 0$. It is a commutative Noetherian Banach algebra equipped with the Gauss norm given by $|\sum_{\alpha\in\mathbb{N}^n} \lambda_\alpha x^\alpha| = \sup_{\alpha \in \mathbb{N}^n} |\lambda_\alpha|$. The unit ball of $K\langle x_1,\ldots,x_n\rangle$ with respect to the Gauss norm is the algebra $\R \langle x_1,\ldots,x_n\rangle$ consisting of power series in $K\langle x_1,\ldots,x_n\rangle$ with all coefficients lying in $\R$.  

A \emph{$K$-affinoid algebra} is any homomorphic image of a Tate algebra, and an \emph{admissible $\R$-algebra} is any homomorphic image of $\R\langle x_1,\ldots,x_n \rangle$ which has no $\R$-torsion. We say that an admissible $\R$-algebra $\A$ is an \emph{affine formal model} in the $K$-affinoid algebra $A$ if $\A$ is a subalgebra of $A$ which spans it as a $K$-vector space.

\begin{defn}[{\cite[Definition 6.1]{DCapOne}}] Let $\A$ be an affine formal model in a $K$-affinoid algebra $A$,  let $L$ be a coherent $(K,A)$-Lie algebra and suppose that $\L$ is an $\A$-submodule of $L$. 
\be \item $\L$ is an $\A$-\emph{lattice} in $L$ if it is finitely generated as an $\A$-module and $K\L=L$. \item $\L$ is an $\A$-\emph{Lie lattice} if in addition it is a sub-$(\R,\A)$-Lie algebra of $L$.\ee \end{defn}

\subsection{Completions of enveloping algebras of Lie--Rinehart algebras}\label{RevComp}

Given an $(R,A)$-Lie algebra $L$ there is an associative $R$-algebra $U(L)$ called the \emph{enveloping algebra} of $L$ with the property that to give an $A$-module $M$ the structure of a left $U(L)$-module is equivalent to giving $M$ the structure of a left module for the $R$-Lie algebra $L$ such that for $a\in A$, $x\in L$ and $m\in M$ 
\begin{equation} \label{LeftModStr}
(ax)m= a(xm), \qmb{and} x(am)=ax(m)+\rho(x)(a)m.
\end{equation}

Similarly, to give a left $A$-module $M$ the structure of a right $U(L)$-module is equivalent to giving $M$ the structure of a right module for the $R$-Lie algebra $L$ such that for $a\in A$, $x\in L$ and $m\in M$ 
\begin{equation} \label{RightModStr}
m(ax)=a(mx)-\rho(x)(a)m \qmb{and} (am)x=a(mx)-\rho(x)(a)m.
\end{equation}

Let $\A$ be an affine formal model in a $K$-affinoid algebra $A$.
\begin{notn}
If $\L$ is an $(\R,\A)$-Lie algebra we write $\h{U(\L)}$ to denote the $\pi$-adic completion of $U(\L)$ and we write $\hK{U(\L)}$ to denote the Noetherian $K$-Banach algebra $K\otimes_\R\h{U(\L)}$.
\end{notn}

\begin{defn}[{\cite[Definition 6.2]{DCapOne}}] Let $L$ be a coherent $(K,A)$-Lie algebra. The \emph{Fr\'echet completion} of $U(L)$ is \[ \w{U(L)}:=\invlim \hK{U(\pi^n\L)} \] for any choice of $\A$-Lie lattice $\L$ in $L$.\end{defn}

As explained in \cite[\S6.2]{DCapOne}, $\w{U(L)}$ is a $K$-Fr\'echet-algebra that does not depend on the choice of affine formal model $\A$ nor on the choice of $\A$-Lie lattice $\L$. A key property of $\w{U(L)}$ used in our work is that it is frequently a \emph{Fr\'echet--Stein algebra} in the sense of \cite{ST}.

\begin{thm}[{\cite[Theorem 6.4]{DCapOne}}] Let $A$ be a $K$-affinoid algebra and let $L$ be a coherent $(K,A)$-Lie algebra. Suppose $L$ has a smooth $\A$-Lie lattice $\L$ for some affine formal model $\A$ in $A$. Then $\w{U(L)}$ is a two-sided Fr\'echet-Stein algebra.\end{thm}

\subsection{The functor \ts{\w{\otimes}}}\label{Revwotimes}

\begin{defn}[{\cite[Definition 7.3]{DCapOne}}] Let $U$ and $V$ be left Fr\'echet--Stein algebras, and let $P$ be a coadmissible left $U$-module. We say that $P$ is a \emph{$U$-coadmissible $(U,V)$-bimodule} if there is a continuous homomorphism $V^{\op}\to \End_U(P)$ with respect to a natural Fr\'echet structure on $\End_U(P)$ defined in \cite[\S7.2]{DCapOne}.\end{defn}

The reason that $U$-coadmissible $(U,V)$-bimodules are useful is that they enable us to base-change coadmissible left $V$-modules to coadmissible left $U$-modules.

\begin{lem}[{\cite[Lemma 7.3]{DCapOne}}] Suppose that $P$ is a $U$-coadmissible $(U,V)$-bimodule. Then for every
coadmissible $V$-module $M$, there is a coadmissible $U$-module \[ P\w{\otimes}_V M\] and a $V$-balanced $U$-linear map
\[ \iota\colon P \times M \to P\w{\otimes}_V M \]
satisfying the following universal property: if $f \colon P \times M \to N$ is a $V$-balanced $U$-linear map with $N$ a coadmissible $U$-module then there is a unique $U$-linear map $g \colon P\w\otimes_V M \to N$ such that $g\iota = f$. Moreover, $P\w\otimes_V M$ is determined by its universal property up to canonical isomorphism.\end{lem}

The base-change functor defined by the Lemma turns out to be associative:
\[ P\w\otimes_V (Q\w\otimes_W M) \cong (P\w\otimes_V Q)\w\otimes_W M\] 
as left $U$-modules, for appropriate choices of $U,V,W,P,Q$ and $M$. By considering opposite algebras we see that there are also right module versions of all these statements.

\subsection{Localisation on affinoid spaces}\label{RevLocaff}

Suppose that $X$ is a $K$-affinoid variety, $\A$ is an affine formal model in $\O(X)$, $\L$ is a smooth $(\R, \A)$-Lie algebra, and $L =  K\otimes_\R \L$.

\begin{defn}[{\cite[Definition 8.1]{DCapOne}}] For each affinoid subdomain $Y$ of $X$, define \[  \w{\sU(L)}(Y) := \w{U(\O(Y)\otimes_{\O(X)} L)}\] \end{defn}

This defines a presheaf of rings on the weak $G$-topology $X_w$ on $X$, and in fact we have the following

\begin{thm}[{\cite[Theorem 8.1]{DCapOne}}] $\w{\sU(L)}$ extends to a sheaf of rings on $X_{\rig}$. \end{thm}

This sheaf of rings can be used to define a localisation functor
\[ \Loc : 
\left\{ 
				\begin{array}{c} 
					co\hspace{-0.1cm}-\hspace{-0.1cm}admissible\\
					\w{U(L)}-\hspace{-0.1cm}modules\hspace{0.1cm}
				\end{array}
\right\}
\to
\left\{ 
				\begin{array}{c} 
					sheaves \hspace{0.1cm}of \\
					\w{\sU(L)}-modules
				\end{array}
\right\}\] 
whose values on every affinoid subdomain $Y$ of $X$ are given by 
\[\Loc(M)(Y)=\w{\sU(L)}(Y)\underset{\w{U(L)}}{\w\otimes}{} M.\]

\begin{thm}[{\cite[Theorem 8.2]{DCapOne}}] $\Loc$ defines a full exact embedding of abelian categories from the category of coadmissible $\w{U(L)}$-modules to the category of sheaves of $\w{\sU(L)}$-modules with vanishing higher \v{C}ech cohomology groups.\end{thm}

\subsection{Lie algebroids}\label{RevLiealg}
Suppose that $X$ is any rigid $K$-analytic space. We recall the notion of a Lie algebroid on $X$.
 
\begin{defn}[{\cite[Definition 9.1]{DCapOne}}] A \emph{Lie algebroid} on $X$ is a pair $(\rho,\sL)$ such that
\begin{itemize} \item $\sL$ is a locally free sheaf of $\O$-modules of finite rank on $X_{\rig}$,
\item $\sL$ has the structure of a sheaf of $K$-Lie algebras, and
\item $\rho \colon\sL\to \T$ is an $\O$-linear map of sheaves of Lie algebras such that
\[ [x, ay] = a[x, y] + \rho(x)(a)y\] whenever $U$ is an admissible open subset of $X$, $x, y\in \sL (U)$ and $a\in \O(U)$.\end{itemize}\end{defn}

\begin{defn}[{\cite[Definition 9.3]{DCapOne}}] Let $\sL$ be a Lie algebroid on the rigid $K$-analytic space $X$, and let $Y$ be an affinoid subdomain of $X$. We say that $\sL(Y)$ \emph{admits a smooth Lie lattice} if there is an affine formal model $\A$ in $\O(Y)$ and a smooth $\A$-Lie lattice $\L$ in $\sL(Y)$. We let $X_w(\sL)$ denote the set of affinoid subdomains $Y$ of $X$ such that $\sL(Y)$ admits a smooth Lie lattice.\end{defn}

By \cite[Lemma 9.3]{DCapOne} $X_w(\sL)$ forms a basis for the $G$-topology on $X$. It enables us to define a `Fr\'echet-completion' of the sheaf $\sU(\sL)$ on $X$ via the following Theorem.

\begin{thm}[{\cite[Theorem 9.3]{DCapOne}}] Let $X$ be a rigid $K$-analytic space. There is a natural functor $\w{\sU(-)}$ from Lie algebroids on $X$ to sheaves of $K$-algebras on $X_{\rig}$ such that there is a canonical isomorphism $\w{\sU(\sL)}|_Y\cong \w{\sU(\sL(Y))}$ for every $Y\in  X_w(\sL)$.\end{thm}

\begin{defn}[{\cite[Definition 9.3]{DCapOne}}] We call the sheaf $\w{\sU(\sL)}$ defined by the Theorem the \emph{Fr\'echet completion} of $\sU(\sL)$. If $X$ is smooth, $\sL = \T$ and $\rho = 1_\T$, we call $\w{\D} := \w{\sU(\T)}$ the Fr\'echet completion of $\D$.\end{defn}

\subsection{Co-admissible sheaves of modules}\label{RevCoad}

Suppose that $\sL$ is a Lie algebroid on a $K$-analytic space $X$. The following defines the notion of \emph{coadmissible $\w{\sU(\sL)}$-modules} that is central to this paper.

\begin{defn} A sheaf of $\w{\sU(\sL)}$-modules $\sM$ on $X_{\rig}$ is \emph{coadmissible} if there is an admissible covering $\{U_i\st i\in I\}$ of $X$ by affinoids in $X_w(\sL)$ such that for each $i\in I$, $\sM(U_i)$ is a coadmissible $\w{\sU(\sL(U_i))}$-module and $\sM|_{U_i}\cong \Loc(\sM(U_i))$.\end{defn}

Co-admissible $\w{\sU(\sL)}$-modules behave well on affinoids in the following sense.

\begin{thm}[{\cite[\S 9.5]{DCapOne}}]  Suppose that $\sL$ is a Lie algebroid on a $K$-affinoid variety $X$ such that $\sL(X)$ admits a smooth Lie lattice. Then $\Loc$ and $\Gamma(X,-)$ are mutually inverse  equivalences of abelian categories
\[
\left\{ 
				\begin{array}{c} 
					co\hspace{-0.1cm}-\hspace{-0.1cm}admissible \\ 
					\w{\sU(\sL)}(X)-\hspace{-0.1cm}modules
				\end{array}
\right\} \cong \left\{
				\begin{array}{c}
				 co\hspace{-0.1cm}-\hspace{-0.1cm}admissible \hspace{0.1cm} sheaves \hspace{0.1cm} of\\ 
				  \w{\sU(\sL)} \hspace{-0.1cm}-\hspace{-0.1cm}modules \hspace{0.1cm} on \hspace{0.1cm} X
				\end{array}
\right\}.
\] 
\end{thm}

\section{Side-switching operations}\label{LRSwitch}

In this section we explain that, as in the classical theory of $\D$-modules, there is an equivalence of categories between the categories of (coadmissible) left $\w{\sU(\sL)}$-modules and of (coadmissible) right $\w{\sU(\sL)}$-modules for any Lie algebroid $\sL$ on a rigid analytic space $X$. 
\subsection{Comparison of left and right modules for Lie--Rinehart algebras}\label{LRequiv}

Suppose that $R$ is a commutative ring and that $A$ is a commutative $R$-algebra and $(L,\rho)$ is an $(R,A)$-Lie algebra.

We recall some well-known constructions; see, for example, \cite[\S 2]{Hue99}. Let $M$ and $N$ be two left $U(L)$-modules, and fix $m \in M$, $n \in N, x \in L$. Analogously to Oda's rule for $\D$-modules, using Equation (\ref{LeftModStr}) above we can make $M \otimes_AN$ into a left $U(L)$-module via the formula
\begin{equation} \label{MAN}
x(m\otimes n)=(xm)\otimes n + m\otimes (xn)
\end{equation}
and $\Hom_A(M,N)$ into a left $U(L)$-module via the formula 
\begin{equation} \label{HomAMNleft}
(xf)(m)=x f(m)- f(xm).
\end{equation}

Similarly, given two right $U(L)$-modules $M$ and $N$ we may make $\Hom_A(M,N)$ into a left $U(L)$-module via the formula 
\begin{equation} \label{HomAMNright}
(xf)(m)=f(mx)-f(m)x. 
\end{equation}

Finally, given a left $U(L)$-module $M$ and a right $U(L)$-module $N$ we may make $N\otimes_A M$ into a right $U(L)$-module via the formula 
\begin{equation}\label{NAM}
(n\otimes m)x=nx\otimes m - n\otimes xm
\end{equation}
and $\Hom_A(M,N)$ into a right $U(L)$-module via the formula 
\begin{equation}\label{HomAMNmixed}
(fx)(m)=f(m)x + f(xm).
\end{equation}
In the last three cases, the given right $L$-action extends to a right $U(L)$-module structure using (\ref{RightModStr}) and the natural left $A$-module structures on $\Hom_A(M,N)$ and $N \otimes_AM$.
\begin{prop} Suppose that $P$ is a right $U(L)$-module that is projective of constant rank $1$ as an $A$-module. Then there are mutually inverse equivalences of categories from left $U(L)$-modules to right $U(L)$-modules given by $P\otimes_A -$ and $\Hom_A(P,-)$.
\end{prop}

\begin{proof} Using formula (\ref{NAM}) above, one can check that the natural functor $P\otimes_A -$ on $A$-modules does indeed induce a functor from left $U(L)$-modules to right $U(L)$ modules. Formula (\ref{HomAMNright}) shows that $\Hom_A(P,-)$ gives a functor in the opposite direction. Moreover the usual natural morphisms of $A$-modules 
\[M\to \Hom_A(P,P\otimes_A M) \qmb{and} P\otimes_A \Hom_A(P,N)\to N\]
are morphisms of $U(L)$-modules when $M$ is a left $U(L)$-module and $N$ is a right $U(L)$-module. The assumption on $P$ implies that they are all isomorphisms.
\end{proof}   
We note that if $L$ is a projective $A$-module of constant rank $d$, then by \cite[Proposition 2.8]{Hue99} the Lie derivative induces the structure of a right $U(L)$-module on the left $A$-module $\Hom_A(\bigwedge^dL,A)$. Since the $A$-module $\Hom_A(\bigwedge^dL,A)$ is projective of constant rank $1$, by the Proposition this condition on $L$ suffices for there to be an equivalence of categories between left $U(L)$-modules and right $U(L)$-modules. 

\begin{examps} \noindent
\be
\item Suppose that $X$ is a smooth irreducible affine variety of dimension $d$ over a field $k$ of characteristic zero and $A = \O(X)$. The $A$-module of derivations $L = \Der_k(A)$ is projective of constant rank $d$, and $U(L) = \D(X)$ is the usual ring of global differential operators on $X$. Then $\Hom_A(\bigwedge^dL,A)$ is the usual dualising module $\Omega(X)$, and in this case the Proposition gives the standard equivalence between left and right $\D(X)$-modules.
\item Returning to full generality, let $L$ be a projective $A$-module and let $U = U(L)$ so that $\gr U \cong \Sym(L)$ by Rinehart's Theorem \cite[Theorem 3.1]{Rinehart}. Suppose that $P$ is a right $U$-module with a single free generator $e$ as an $A$-module and that $I$ is an ideal in $A$. Then there is an isomorphism \[ U/IU\to P\otimes_A U/UI\] of right $U$-modules given by $\alpha+IU\mapsto (e\otimes 1)\alpha$.
This can be seen by giving both sides their natural filtrations and observing that the induced morphism between associated graded modules is the isomorphism of symmetric algebras $
\Sym(L/IL)\to \Sym(L/IL)$ given by multiplication by $(-1)^i$ in degree $i$. 
\ee \end{examps}

\subsection{An involution on \ts{P \otimes_A U(L)}} \label{PAUL}
Suppose that $P$ is a projective module of constant rank $1$ with the structure of a right $U(L)$-module. Then there are two natural ways to view $P\otimes_A U(L)$ as a right $U(L)$-module: one of these comes from the left action of $U(L)$ on itself and formula (\ref{NAM}) above; the other, which we'll denote by $\circ$, comes from the right action of $U(L)$ on itself. We'll write $P\oslash_AU(L)$ to denote the left $A$-module $P\otimes_A U(L)$ equipped with the $\circ$-action of $U(L)$ on the right. The following Lemma can be viewed as saying that there is an automorphism $\alpha$ of the $A$-module $P\otimes_A U(L)$ that exchanges these two structures. 

\begin{lem} Let $\alpha\colon P\otimes_A U(L)\to P\otimes_A U(L)$ be defined by $\alpha( p \otimes u ) = (p \otimes 1)u.$ Then 
\be \item $\alpha(t\circ u)=\alpha(t)u$ for all $t \in P \otimes_A U(L)$ and $u \in U(L)$,
\item $\alpha( (p \otimes u)v ) = (p \otimes v)u$ for all $p \in P$ and $u,v \in U(L)$, and
\item $\alpha^2 = \id$.
\ee
\end{lem}
\begin{proof} (a) We may assume that $t = p \otimes v$ for some $p \in P$ and $v \in U(L)$. Then
\[\alpha((p \otimes v)\circ u) = \alpha(p \otimes vu) = (p \otimes 1)vu = \alpha(p \otimes v)u.\]
(b) There is a natural exhaustive positive filtration $F_\bullet$ on $U(L)$ such that $F_0 = A$ and $F_n = L \cdot F_{n-1}+F_{n-1}$ for all $n \geq 1$. Proceed by induction on the filtration degree $n$ of $v \in U(L)$. When $n = 0$, $v \in A$ and $\alpha( (p \otimes u)v ) = \alpha( pv \otimes u ) = (pv \otimes 1)u = (p \otimes v)u$
 as claimed. Suppose now that $v = xw$ for some $x \in L$ and $w \in U(L)$ of degree strictly less than $n$; since the statement is linear in $v$ this case suffices. Then
\[\alpha( (p \otimes u) v ) = \alpha( ((p \otimes u)x) w)=  \alpha( (px \otimes u - p \otimes xu)w ) \]
which by the inductive hypothesis is equal to 
\[(px \otimes w)u - (p \otimes w)(xu)= (px \otimes w - (p \otimes w)x)u = (p \otimes xw)u = (p \otimes v)u.\]
(c) From part (b) we have $\alpha^2(p \otimes u) = \alpha(( p\otimes 1)u) = (p \otimes u)1 = p \otimes u$. 
\end{proof}
It follows immediately that $\alpha$ is invertible and that 
\[\alpha(tu) = \alpha(t) \circ u \qmb{for all} t \in P \otimes_A U(L)\qmb{and} u \in U(L).\]

\subsection{Comparison of left and right \ts{\hK{U(\L)}}-modules}\label{LRequivBan}

Suppose now that $\A$ is an affine formal model in a $K$-affinoid algebra $A$ and $\L$ is a smooth $(\R,\A)$-Lie algebra of constant rank $d$. Let $\Omega_{\L}:= \Hom_{\A}(\bigwedge^d\L,\A)$. This is a projective $\A$-module of constant rank $1$ with the structure of a right $U(\L)$-module by the discussion in Subsection \ref{LRequiv} above.

\begin{lem} If $\M$ is a $\pi$-adically complete $\A$-module then $\Omega_\L\otimes_\A \M$ and $\Hom_\A(\Omega_\L, \M)$ are also $\pi$-adically complete.
\end{lem}
\begin{proof} $\Omega_L$ is a direct summand of some free $\A$-module $\A^r$. Then $\Omega_\L\otimes_\A \M$ and $\Hom_{\A}(\Omega_\L,\M)$ may be viewed as direct summands of $\M^r$. Since finite direct sums and direct summands of $\pi$-adically complete modules are $\pi$-adically complete, the result follows.
\end{proof}

\begin{prop} There is an equivalence of categories between the category of $\pi$-adically complete left $\h{U(\L)}$-modules and the category of $\pi$-adically complete right $\h{U(\L)}$-modules given by $\Omega_{\L}\otimes_{\A}-$. Moreover this functor restricts to an equivalence of categories between finitely generated left $\h{U(\L)}$-modules and finitely generated right $\h{U(\L)}$-modules.
\end{prop}

\begin{proof} It follows from the Lemma that the equivalence of categories $\Omega_{\L}\otimes_{\A}-$ from left $U(\L)$-modules to right $U(\L)$-modules given by Proposition \ref{LRequiv} restricts to an equivalence of categories between $\pi$-adically complete left $U(\L)$-modules and $\pi$-adically complete right $U(\L)$-modules.

We will show that restriction along $U(\L)\to \h{U(\L)}$ defines equivalences of categories from the category of $\pi$-adically complete left (respectively, right) $\h{U(\L)}$-modules and $\pi$-adically complete left (respectively, right) $U(\L)$-modules. Certainly the restriction functors are faithful so we must show that they are also full and essentially surjective on objects. Suppose that $f\colon \M\to \N$ is a $U(\L)$-linear map between $\pi$-adically complete $\h{U(\L)}$-modules. We must show that $f$ is also $\h{U(\L)}$-linear. But $f$ induces $U(\L)/\pi^n U(\L)$-linear morphisms $f_n\colon \M/\pi^n \M\to \N/\pi^n \N$ for each $n\ge 0$ and we may identify $f$ with $\invlim f_n\colon \M\to \N$ which is $\h{U(\L)}$-linear as required. Similarly, if $\M$ is any $\pi$-adically complete $U(\L)$-module, then $\M/\pi^n\M$ is naturally a $U(\L)/\pi^n U(\L)$-module for each $n\ge 0$ and so $\M\cong \invlim \M/\pi^n \M$ is a naturally a $\h{U(\L)}$-module as required. 

We now see that $\Omega_{\L}\otimes_{\A}-$ is an equivalence of categories between $\pi$-adically complete left $\h{U(\L)}$-modules and $\pi$-adically complete right $\h{U(\L)}$-modules. Now, every finitely generated $\h{U(\L)}$-module is $\pi$-adically complete, and $\Omega_{\L}\otimes_{\A}-$ preserves inclusions between finitely generated modules. So if $\M$ is a Noetherian left $\h{U(\L)}$-module, then an ascending chain of finitely generated right $\h{U(\L)}$-submodules of $\Omega_{\L}\otimes_{\A} \M$ corresponds to an ascending chain of left submodules of $\M$, which terminates. So $\Omega_{\L}\otimes_{\A}\M$ is a Noetherian right $\h{U(\L)}$-module.\end{proof}
Let $\Omega_L:= K\otimes_\R \Omega_\L$. For any $\A$-module $\M$, the natural $A$-module isomorphism 
\[ \Omega_L \otimes_A (K \otimes_\R \M) \cong K \otimes_\R (\Omega_\L \otimes_\A \M)\]
induces a left $\hK{U(\L)}$-module structure on $\Omega_L \otimes_A (K \otimes_\R \M)$ whenever $\M$ is a left $\h{U(\L)}$-module. Similary, the natural $A$-module isomorphism
\[\Hom_A(\Omega_L, K \otimes_\R\M) \cong K \otimes_\R \Hom_\A(\Omega_\L,\M)\]
induces a right $\hK{U(\L)}$-module structure on $\Hom_A(\Omega_L, K \otimes_\R \M)$ whenever $\M$ is a right $\h{U(\L)}$-module. These constructions are functorial in $\M$.
\begin{thm} The functors $\Omega_L\otimes_A-$ and $\Hom_A(\Omega_L,-)$ are a mutually inverse pair of equivalences of categories between finitely generated left $\hK{U(\L)}$-modules and finitely generated right $\hK{U(\L)}$-modules.
\end{thm}
\begin{proof} Suppose that $M$ is a finitely generated left $\hK{U(\L)}$-module. Then we can find a finitely generated $\h{U(\L)}$-submodule $\M\subset M$ such that $K\otimes_\R \M\cong M$. Therefore  $\Omega_L\otimes_A M\cong K\otimes_\R (\Omega_\L\otimes_\A \M)$ is a finitely generated right $\hK{U(\L)}$-module by the Proposition. The same argument shows that $\Hom_A(\Omega_L,-)$ sends finitely generated right $\hK{U(\L)}$-modules to finitely generated left $\hK{U(\L)}$-modules. 

The functor $\Omega_L \otimes_A -$ is left adjoint to $\Hom_A(\Omega_L,-)$ on $A$-modules, and the unit and counit morphisms for this adjunction are isomorphisms on all modules of the form $\M \otimes_\R K$ by the Proposition. So the restrictions of these functors to finitely generated $\hK{U(\L)}$-modules are mutually inverse equivalences.
\end{proof}

\subsection{Comparison of left and right coadmissible \ts{\w{U(L)}}-modules}\label{LRequivCoad}

Suppose now that $A$ is a $K$-affinoid algebra and $L$ is a smooth $(K,A)$-Lie algebra of constant rank $d$. Suppose that $L$ admits a smooth $\A$-Lie lattice $\L$ for some affine formal model $\A$ in $A$, and recall the Fr\'echet--Stein algebra $\w{U(L)}$ from \cite[Theorem 6.4]{DCapOne}. As in Subsection \ref{LRequivBan}, let $\Omega_L$ denote $\Hom_A(\bigwedge^dL,A)$. 

Recall the isomorphism $\alpha\colon \Omega_L\otimes_A U(L)\to \Omega_L\oslash_A U(L)$ from Lemma \ref{PAUL} that switches the two natural right $U(L)$-module structures on $\Omega_L\otimes_A U(L)$.

\begin{lem} There is a commutative diagram \[ \xymatrix{ \Omega_L\otimes_A \hK{U(\pi\L)} \ar[r] \ar[d] & \Omega_L\otimes_A \hK{U(\L)} \ar[d] \\ 
                                                                 \Omega_L\oslash_A \hK{U(\pi\L)} \ar[r]        & \Omega_L\oslash_A \hK{U(\L)} }\] of right $\hK{U(\pi\L)}$-modules with horizontal arrows given by inclusions and vertical arrows continuous extensions of $\alpha$.
\end{lem}

\begin{proof}  

By Lemma \ref{PAUL}, there is an isomorphism $\alpha: \Omega_\L\otimes_\A U(\L)\to \Omega_\L\oslash_\A U(\L)$ of right $U(\L)$-modules. We may $\pi$-adically complete this map and then invert $\pi$ in order to obtain $\hK{\alpha}$, the unique continuous extension of $\alpha$ to an isomorphism \[ \hK{\alpha}\colon \Omega_L\otimes_A \hK{U(\L)}\to \Omega_L\oslash_A \hK{U(\L)} .\] The result follows immediately.   \end{proof}

\begin{thm} The pair $\Omega_L\otimes_A -$ and $\Hom_A(\Omega_L,-)$ define mutually inverse equivalences of categories between coadmissible left $\w{U(L)}$-modules and coadmissible right $\w{U(L)}$-modules.
\end{thm}

\begin{proof} Let $\L$ be a smooth $\A$-Lie lattice in $L$, and write $U_n=\hK{U(\pi^n\L)}$ and $U=\w{U(L)}$. Suppose that $(M_\bullet)$ is a family of finitely generated left $U_\bullet$-modules and let $N_\bullet:=\Omega_L\otimes_A M_\bullet$. By Theorem \ref{LRequivBan}, $(N_\bullet)$ is a family of finitely generated right $U_\bullet$-modules. We will verify that $(N_\bullet)$ is coherent if and only if $(M_\bullet)$ is coherent, that is, that there are isomorphisms $N_{n+1}\otimes_{U_{n+1}}U_n\cong N_n$ for each $n\ge 0$ if and only if there are isomorphisms $U_n\otimes_{U_{n+1}}M_{n+1}\cong M_n$ for each $n\ge 0$. 

For each finitely generated left $U_{n+1}$-module $Q$, define a right $U_n$-linear map
\[\theta_{Q}\colon (\Omega_L\otimes_A Q)\otimes_{U_{n+1}} U_n\to \Omega_L\otimes_A (U_n\otimes_{U_{n+1}}Q)\]
by setting $\theta_{Q} \left((\omega\otimes m)\otimes r\right) = (\omega\otimes 1\otimes m)r$. Then $\theta$ is a natural transformation between two right exact functors, and it follows from the Lemma that $\theta_{U_{n+1}}$ as an isomorphism. Hence $\theta_Q$ is an isomorphism for all $Q$ by the Five Lemma. Thus $N_{n+1}\otimes_{U_{n+1}}U_n\cong \Omega_L\otimes_A (U_n\otimes_{U_{n+1}}M_{n+1})$. So $N_n\cong N_{n+1}\otimes_{U_{n+1}}U_n$ if and only if $M_n\cong U_n\otimes_{U_{n+1}}M_{n+1}$.

It now follows from Theorem \ref{LRequivBan} that $(M_\bullet) \mapsto (N_\bullet)$ is an equivalence of categories between coherent sheaves of left $U_\bullet$-modules and coherent sheaves of right $U_\bullet$-modules. Finally, since $\Omega_L$ is a direct summand of a free $A$-module, for every coadmissible left $U$-module $M$ there are canonical isomorphisms
\[ \Omega_L \otimes_A M \cong \Omega_L \otimes_A(\invlim U_n \otimes_U M) \cong \invlim \Omega_L \otimes_A (U_n \otimes_U M) \]
of left $A$-modules. Using the composition of these isomorphisms we can define a right $U$-module structure on $\Omega_L \otimes_A M$. Similarly, the canonical isomorphisms
\[ \Hom_A(\Omega_L, N) \cong \Hom_A(\Omega_L, N \otimes_U U_n) \cong U_n\otimes_U \Hom_A(\Omega_L,N)\]
induce a left $U$-module structure on $\Hom_A(\Omega_L, N)$ for every coadmissible right $U$-module $N$. The result now follows from \cite[Corollary 3.3]{ST}.
\end{proof}

\subsection{Comparison of left and right coadmissible \ts{\w{\sU(\sL)}}-modules}\label{wULswitch}
Suppose now that $X$ is a rigid analytic space and that $\sL$ is a Lie algebroid on $X$ with constant rank $d$. Then we have at our disposal the invertible sheaf
\[\Omega_{\sL}:=\mathpzc{Hom}_{\O}(\bigwedge^d \sL, \O).\]
\begin{lem} Suppose that $Z\subset Y$ are open affinoid subvarieties of $X$ such that $\sL(Y)$ admits a smooth Lie lattice $\L$ for some affine formal model $\A$ in $\O(Y)$ and $\B$ is an $\L$-stable affine formal model in $\O(Z)$. Then there is a commutative diagram of $\hK{U(\L)}$-modules \[ \xymatrix{ \Omega_{\sL}(Y)\underset{\O(Y)}{\otimes} \hK{U(\L)} \ar[r] \ar[d] & \Omega_{\sL}(Z)\underset{\O(Z)}{\otimes}\hK{U(\B\otimes_\A \L)} \ar[d] \\ \Omega_{\sL}(Y)\underset{\O(Y)}{\oslash}\hK{U(\L)}\ar[r] & \Omega_{\sL}(Z)\underset{\O(Z)}{\oslash}  \hK{U(\B\otimes_\A \L)}}\] whose vertical maps are isomorphisms.
\end{lem}
\begin{proof} We construct the vertical arrows by following the proof of Lemma \ref{LRequivCoad}. The horizontal arrows are induced by the functoriality of $\hK{U(-)}$. It is straightforward to use Lemma \ref{PAUL} to verify that the diagram commutes.
\end{proof}

Recall the sheaf $\sU := \w{\sU(\sL)}$ from Subsection \ref{RevLiealg}.

\begin{prop} Suppose that $Y$ is an open affinoid subvariety of $X$ such that $\sL(Y)$ admits a smooth Lie lattice.
\be \item Let $M$ be a coadmissible left $\w{U(\sL(Y))}$-module. Then there is an isomorphism of  right $\sU|_{Y_w}$-modules 
\[ \Loc\left(\Omega_{\sL}(Y)\otimes_{\O(Y)}M\right)\cong  \Omega_{\sL|Y_w}\otimes_{\O_Y}\Loc(M).\] 
\item Let $N$ be a coadmissible right $\w{U(\sL(Y))}$-module. Then there is an isomorphism of left $\sU|_{Y_w}$-modules 
\[ \Loc\left(\Hom_{\O(Y)}(\Omega_{\sL}(Y),N)\right)\cong \mathpzc{Hom}_{\O_Y}\left( \Omega_{\sL|Y_w}, \Loc(N)\right).\] 
\ee
\end{prop}
\begin{proof} Let $Z$ be an affinoid subdomain of $Y$.  Consider $\A$ an affine formal model in $\O(Y)$ such that $\sL(Y)$ admits a smooth $\A$-Lie lattice $\L$. By rescaling $\L$ if necessary and applying \cite[Lemma 7.6(b)]{DCapOne}, we may assume that $\O(Z)$ admits an $\L$-stable affine formal model $\B$. Write $U_n:=\hK{U(\pi^n\L)}$ and $V_n:= \hK{U(\B\otimes_\A\pi^n\L)}$ so that $\sU(Y) = \invlim U_n$ and $\sU(Z) = \invlim V_n$. For each finitely generated left $U_n$-module $Q$, define a right $V_n$-linear map
\[ \tau_Q\colon  (\Omega_{\sL}(Y)\otimes_{\O(Y)} Q)\otimes_{U_n}V_n\to  \Omega_{\sL}(Z)\otimes_{\O(Z)} (V_n\otimes_{U_n}Q)\]
by setting $\tau_Q\left((\omega\otimes m)\otimes v\right)= (\omega|_Z\otimes (1 \otimes m))v$. Then $\tau$ is a natural transformation between two right exact functors, and using the Lemma we see that $\tau_{U_n}$ is an isomorphism. Hence $\tau_{Q}$ is an isomorphism for all $Q$ by the Five Lemma. 

Now there is a natural isomorphism 
\[  \left(\Omega_{\sL}(Y)\underset{\O(Y)}{\otimes}M\right)\underset{\sU(Y)}{\w{\otimes}}\sU(Z)\cong \Omega_{\sL}(Z)\underset{\O(Z)}{\otimes} \left(\sU(Z)\underset{\sU(Y)}{\w\otimes}M\right)\] of right $\sU(Z)$-modules, and the isomorphism
\[\Loc\left(\Omega_{\sL}(Y)\otimes_{\O(Y)}M\right)\cong  \Omega_{\sL|Y_w}\otimes_{\O_Y}\Loc(M)\]
follows. This proves part (a), and part (b) has a similar proof.
\end{proof}

\begin{cor} Keep the notation of the Proposition.
\be \item $\Omega_{\sL|Y_w} \otimes_{\O_Y} \Loc(M)$ is a coadmissible right $\sU|_{Y_w}$-module.
\item $\mathpzc{Hom}_{\O_Y}\left( \Omega_{\sL|Y_w}, \Loc(N)\right)$ is a coadmissible left $\sU|_{Y_w}$-module. 
\ee
\end{cor}

We are now ready to prove the main Theorem in this section.
\begin{thm} There is a mutually inverse pair of equivalences of categories $\Omega_{\sL}\otimes_{\O_X}-$ and $\mathpzc{Hom}_{\O_X}(\Omega_{\sL},-)$ between coadmissible left $\w{\sU(\sL)}$-modules and coadmissible right $\w{\sU(\sL)}$-modules. 
\end{thm} 

\begin{proof} The statement is local so by \cite[Lemma 9.3]{DCapOne} we may reduce to the case where $X$ is affinoid and $\sL(X)$ admits a smooth Lie lattice. This case follows from the Corollary and Theorem \ref{LRequivCoad}.
\end{proof}

\textbf{Unless explicitly stated otherwise, until the end of Section \ref{Pullpush} the term "module" will mean \emph{right} module.}

\section{Kashiwara's equivalence for \ts{\hK{U(\L)}}}\label{ConthKUL}
From now on, we assume that the characteristic of our ground field $K$ is zero. 
\subsection{Centralisers in \ts{L}}\label{Cent}

Let $R$ be a commutative base ring. Suppose that  $\varphi\colon A\to B$ is an injective homomorphism of commutative $R$-algebras, $(L,\rho)$ is an ($R$,$A$)-Lie algebra, and $\sigma\colon L\to \Der_R(B)$ is an $A$-linear Lie algebra homomorphism such that
\[ \sigma(x) \circ \varphi = \varphi \circ \rho(x) \qmb{for all} x \in L.\]
For every subset $F$ of $B$, we define 
\[C_L(F) := \{y \in L \st \sigma(y) \cdot f = 0 \qmb{for all} f\in F\} \] to be the \emph{centraliser of $F$ in $L$}. It is straightforward to verify that $C_L(F)$ is always an ($R$,$A$)-Lie subalgebra of $L$. We will abuse notation and simply write $x \cdot b$ to mean $\sigma(x)(b)$ if $x \in L$ and $b \in B$.

\begin{lem} Suppose that $f_1,\ldots, f_r\in B$ are such that $L\cdot f_i\subset A$ for each $1\le i\le r$ and there exist $x_1,\ldots x_r\in L$ with $x_i\cdot f_j=\delta_{ij}$. Then 
\[L = \left(\bigoplus_{i=1}^r A x_i\right) \oplus C_L(\{f_1,\ldots,f_r\}).\] In particular, if $L$ is smooth then $C_L(\{f_1,\ldots,f_r\})$ is smooth. 
\end{lem}
\begin{proof} Write $C=C_L(\{f_1,\ldots,f_r\})$. If $u = \sum_{i=1}^r a_ix_i \in C$ for some $a_i \in A$ then for each $j$ we can compute $0=u \cdot f_j = \sum_{i=1}^r a_i (x_i \cdot f_j) = a_j$. Hence $u = \sum_{i=1}^r a_ix_i = 0$ and the sum $\sum_{i=1}^r A x_i + C$ is direct.  On the other hand, if $u \in L$ then \[(u - \sum_{i=1}^r (u\cdot f_i)x_i) \cdot f_j = u \cdot f_j - \sum_{i=1}^r (u\cdot f_i) (x_i \cdot f_j) = 0.\] Hence $u = \sum_{i=1}^r(u\cdot f_i)x_i + (u - \sum_{i=1}^r (u\cdot f_i)x_i) \in \sum_{i=1}^r A x_i + C$. \end{proof}

\subsection{The submodules $M[F]$ and $M_{\Dp}(F)$}\label{MdpF}
Until the end of Section \ref{ConthKUL}, we fix an affine formal model $\A$ in a $K$-affinoid algebra $A$ and a smooth $(\R, \A$)-Lie algebra $\L$.

\begin{defn} Let $M$ be a finitely generated  $\hK{U(\L)}$-module and $F\subset A$.
\be
\item $M[F] := \{m \in M : mf = 0\mbox{ }\forall f\in F\}$.
\item $M_{\Dp}(F) := \{ m \in M : \lim\limits_{n\to \infty} m\frac{f^n}{n!}  = 0\mbox{ }\forall f\in F\}$.
\ee\end{defn}

Let $\C = C_{\L}(F)$. Because the elements of $F$ are central in $\hK{U(\C)}$, we see that $M[F] \subseteq M_{\Dp}(F)$ are $\hK{U(\C)}$-submodules of $M$. We will soon see that if $\L\cdot F\subset \A$ then $M_{\Dp}(F)$ is even a $\hK{U(\L)}$-submodule of $M$. 

\begin{prop} The functor $(-)[F]$ from $\hK{U(\L)}$-modules to $\hK{U(\C)}/(F)$-modules is right adjoint to $-\otimes_{\hK{U(\C)}}\hK{U(\L)}$.\end{prop}
\begin{proof} Suppose that $M$ is a $\hK{U(\C)}/(F)$-module and $N$ is a $\hK{U(\L)}$-module. By the universal property of $\otimes$ there is a natural isomorphism \[ \Hom_{\hK{U(\L)}}(M\otimes_{\hK{U(\C)}}\hK{U(\L)},N)\cong \Hom_{\hK{U(\C)}}(M,N). \] Since $MF=0$ and $F$ is central in $\hK{U(\C)}$, any $\hK{U(\C)}$-linear map from $M$ to $N$ will have image in $N[F]$. Thus there is also a natural isomorphism \[ \Hom_{\hK{U(\C)}}(M,N)\cong \Hom_{\hK{U(\C)}}(M,N[F]).\]Putting these two isomorphisms together we obtain the result. 
\end{proof}

 It is clear that $(-)\otimes_{\hK{U(\C)}}\hK{U(\L)}$ sends finitely generated $\hK{U(\C)}$-modules to finitely generated $\hK{U(\L)}$-modules. We'll see that if $F=\{f_1,\ldots,f_r\}$ is a finite subset of $A$ such that $\L\cdot (f_1,\ldots,f_r)=\A^r$ then $(-)[F]$ also sends finitely generated modules to finitely generated modules. Moreover we'll show that under this hypothesis, when the adjunction is restricted to finitely generated modules, its unit is an isomorphism and its counit isomorphic to the inclusion $M_{\Dp}(F)\to M$.

If $\M$ is any $\R$-module, we'll write $\M_k:=\M\otimes_\R k$ for its reduction modulo $\pi$ in what follows.

\subsection{Proposition}\label{Lieflat} 
Let $\C$ be a sub-$(\R,\A$)-Lie algebra of $\L$, and suppose that $\C$ has an $\A$-module complement in $\L$. Then $\hK{U(\L)}$ is a faithfully flat $\hK{U(\C)}$-module.
\begin{proof} The assumptions on $\C$ and $\L$ force $\C$ to be a projective $\A$-module. Now, with respect to the $\pi$-adic filtrations on $\h{U(\L)}$ and $\h{U(\C)}$ respectively, 
\[\gr \hK{U(\L)}\cong k[t,t^{-1}]\otimes_k U(\L_k) \qmb{and} \gr \hK{U(\C)}\cong k[t,t^{-1}]\otimes_k U(\C_k).\]
These algebras carry a natural positive filtration with $t$ and $\A_k$ in degree zero, and $\L_k$, $\C_k$ in degree one. Since $\L_k$ and $\C_k$ are smooth, the respective associated graded rings are $k[t,t^{-1}]\otimes_k\Sym(\L_k)$ and $k[t,t^{-1}]\otimes_k\Sym(\C_k)$ by \cite[Theorem 3.1]{Rinehart}. The map $\hK{U(\C)}\to \hK{U(\L)}$ induces the natural inclusion $k[t,t^{-1}]\otimes_k \Sym(\C_k)\to k[t,t^{-1}]\otimes_k \Sym(\L_k)$ which is faithfully flat, because $\C_k$ has an $\A_k$-module complement in $\L_k$ by assumption. We can now apply \cite[Chapter II, Proposition 1.2.2]{LVO} twice.
\end{proof}

\begin{cor} Let $F = \{f_1,\ldots, f_r\} \subset A$ be such that $\L \cdot (f_1,\ldots,f_r) = \A^r$, and let $\C = C_{\L}(F)$. Then $\hK{U(\L)}$ is a faithfully flat $\hK{U(\C)}$-module.
\end{cor}
\begin{proof} This follows from Lemma \ref{Cent} and the Proposition.
\end{proof}

\subsection{Construction of a lattice that is stable under divided powers}\label{DivPow}

For the remainder of this section we'll write $U=\hK{U(\L)}$ and $\U=\h{U(\L)}$. Suppose that $f\in A$ is such that $\L\cdot f\subset \A$ and that $M$ is a finitely generated $U$-module. Let $N=M_{\Dp}(f)\cdot U$. Since $U$ is Noetherian, $N$ is also a finitely generated $U$-module and we may fix a finite generating set $\{v_1,\ldots,v_s\}\subset M_{\Dp}(f)$ for $N$ as a $U$-module. We say that a $\U$-submodule of $N$ is a \emph{$\U$-lattice} in $N$ if it is finitely generated over $\U$ and generates $N$ as a $K$-vector space. Thus in particular,
\[\M_0 := \sum_{j=1}^s v_j \U \]
is a $\U$-lattice in $N$.
\begin{lem} There exists an integer $t$ such that
\[ v_j\left(\frac{f^{\alpha_1}}{\alpha_1!}\right)^{\beta_1} \cdots \left(\frac{f^{\alpha_m}}{\alpha_m!}\right)^{\beta_m}  \in \pi^{-t} \M_0 \]
for all $\alpha,\beta\in \mathbb{N}^m$ and all $j = 1,\ldots, s$.
\end{lem}
\begin{proof} Since each $v_j$ lies in $M_{\Dp}(f)$ by assumption, $v_jf^n/n! \to 0$ as $n \to \infty$ for each $j$. So each of these sequences is bounded in $N$ and is therefore contained in $\pi^{-t}\M_0$ for some $t \geq 0$. Let $n = \sum_{i=1}^m \alpha_i\beta_i$. Then 
\[ v_j\left(\frac{f^{\alpha_1}}{\alpha_1!}\right)^{\beta_1} \cdots \left(\frac{f^{\alpha_m}}{\alpha_m!}\right)^{\beta_m}  = \frac{ n! } { \alpha_1!^{\beta_1} \cdots \alpha_m!^{\beta_m}} \frac{v_jf^n}{n!}\in \pi^{-t}\M_0\]
because $\frac{ n! } { \alpha_1!^{\beta_1} \cdots \alpha_m!^{\beta_m}}$ is a multinomial coefficient and is therefore an integer.
\end{proof}

\begin{prop} There is at least one $\U$-lattice $\M$ in $M_{\Dp}(f)\cdot U$ which is stable under the action of all of the divided powers $f^i/i!$.
\end{prop}
\begin{proof} Let $N=M_{\Dp}(f) \cdot U$ and let $\B$ be the subalgebra of $A$ generated by $\A$ and $\{f^i/i! : i \geq 0\}$. Since $\L\cdot f\subset \A$ by assumption, an easy induction on $i$ shows that $y \cdot \frac{f^i}{i!} = \frac{f^{i-1}}{(i-1)!}(y\cdot f)\in \B$ for all $y \in \L$ and all $i \geq 0$. So, as also $\L \cdot \A \subset \B$ we see that $\L \cdot \B \subset \B$. Thus the action of $\L$ on $\A$ lifts to $\B$, so we can form the $(\R, \B)$-Lie algebra $\B \otimes_{\A} \L$ by \cite[Lemma 2.2]{DCapOne}.  Now $T := U(\B \otimes_{A} \L) \cong \B\otimes_{\A}U(\L)$ by \cite[Proposition 2.3]{DCapOne}, so $T$ is generated as a $U(\L)$-module by all possible monomials 
\[ \left(\frac{f^{\alpha_1}}{\alpha_1!}\right)^{\beta_1} \cdots \left(\frac{f^{\alpha_m}}{\alpha_m!}\right)^{\beta_m} \in \B, \quad \alpha,\beta \in \mathbb{N}^m.\]
Since $\M_0$ is a $U(\L)$-submodule of $N$, by the Lemma we can find an integer $t$ such that 
\[    v_jT \subseteq \pi^{-t} \M_0 \qmb{for all} j=1,\ldots, s.\]
Let $\M$ be the closure of $\sum_{j=1}^s v_jT$ in $N$. Since $\U$ is the $\pi$-adic completion of the subalgebra $U(\L)$ of $T$, $\M$ is a $\U$-submodule of $N$ containing $v_1,\ldots,v_s$. Thus
\[ \M_0 \subseteq \M \subseteq \pi^{-t}\M_0\]
because $\pi^{-t}\M_0$ is closed in $N$. Since $\pi^{-t}\M_0$ is a finitely generated module over the Noetherian ring $\U$, $\M$ is finitely generated over $\U$. Therefore it is a $\U$-lattice in $N$ which is stable under all divided powers $f^i/i!$ by construction.
\end{proof}

\subsection{Proposition} \label{DpUSubmod} Suppose that $F\subset A$ is such that $\L\cdot F\subset \A$, and $M$ is a finitely generated $U$-module. Then $M_{\Dp}(F)$ is a $U$-submodule of $M$. 
\begin{proof} Because $M_{\Dp}(F)=\cap_{f\in F}M_{\Dp}(f)$, we may assume that  $F=\{f\}$. Using Proposition \ref{DivPow}, choose a $\U$-lattice $\M$ in $M_{\Dp}(f)\cdot U$ stable under all $f^i/i!$ and write $\M_{\Dp}(f)=M_{\Dp}(f)\cap \M$. Certainly $\M_{\Dp}(f)$ is an $\A$-submodule of $M$. 

Suppose that $x\in \L$, $m\in \M_{\Dp}(f)$ and $n\ge 0$. Then \[ (mx)\frac{f^n}{n!}=m\frac{f^n}{n!}x+m(x\cdot f)\frac{f^{n-1}}{(n-1)!}. \] Since $m\in M_{\Dp}(f)$, and $x\cdot f\in \A$ and $x\in \L$ preserve the lattice $\M$ of $M$, we see that $mx\in \M_{\Dp}(f)$. Thus we may view $\M_{\Dp}(f)$ as a $U(\L)$-submodule of $M$. 

Next we show that $\M_{\Dp}(f)$ is $\pi$-adically closed. Suppose that $m_n$ is a sequence in $\M_{\Dp}(f)$ converging  $\pi$-adically to $m\in \M$. Then for all $r>0$ there is $n_0$ such that $m-m_{n_0}\in \pi^r\M$ and there is $n_1\ge n_0$ such that $m_{n_0}f^n/n!\in \pi^r\M$ for all $n>n_1$. Thus $mf^n/n!\in \pi^r\M$ for all $n>n_1$ because $\M$ is stable under all $f^i/i!$. So $m\in \M_{\Dp}(f)$ as required.

It follows that $\M_{\Dp}(f)$ is a $\U$-module, so $M_{\Dp}(f)$ is a $U$-module as required. \end{proof}

\subsection{Reduction mod \ts{\pi}}\label{CounitModPi} 

Suppose now in addition to $\L \cdot f \subset \A$ that there is an element $x\in \L$ such that $x \cdot f = 1.$ This is equivalent to assuming that $\L \cdot f = \A$. Let $M$ be a finitely generated $U$-module, and write $\C := C_\L(f)$ and $\V:=\h{U(\C)}$. 

\begin{lem} Let $\M$ be a $\U$-lattice in $M_{\Dp}(f)$ stable under all divided powers in $f$, and let $\bar{x}$ denote the image of $x \in \L$ in $\L_k$. 
\be
\item $\U_k$ is a free left $\V_k$-module with basis $\{\bar{x}^j : j \geq 0\}$.
\item The natural map $ \M[f]_k\otimes_{\V_k} \U_k \to \M_k$ is injective.
\ee 
\end{lem}
\begin{proof} (a) By Lemma \ref{Cent}, $\L = \A x \oplus \C$. Therefore $\L_k = \A_k \bar{x} \oplus \C_k$ so the powers of $\bar{x}$ in the symmetric algebra $\Sym(\L_k)$ form a homogeneous basis for $\Sym(\L_k)$ as a graded left $\Sym(\C_k)$-module. Since $\U_k = U(\L_k)$ and $\V_k = U(\C_k)$, we can apply  \cite[Theorem 3.1]{Rinehart}.

(b) Let $\xi \in \M[f]_k  \otimes_{\V_k} \U_k$ map to zero in $\M_k$. By part (a), we can write $\xi$ uniquely in the form $\xi = \sum_{j=0}^n \overline{q_j}\otimes \bar{x}^j$ for some $q_j \in \M[f]$. Now 
\[ \sum_{j=0}^n  q_jx^j \in \pi \M, \qmb{so} \sum_{j=0}^n  q_jx^j \frac{f^n}{n!}  \in \pi \M\]
because the lattice $\pi \M$ is stable under $f^n/n!$ by construction. Now 
\[ q_jx^jf = q_j[x^j,f] + q_jfx^j = jq_jx^{j-1} \qmb{for each} j\]
because $[x,f] = x\cdot f = 1$ and $q_j f= 0$. Therefore
\[  \sum_{j=0}^n  q_j x^j\frac{f^n}{n!}= \sum_{j=0}^n  \binom{j}{n} q_jx^{j-n} = q_n \in \pi \M \cap \M[f].\]
But $\pi \M \cap \M[f] = \pi \M[f]$ because $\M$ is $\pi$-torsion-free, so $\overline{q_n} = 0$. Continuing like this, we see that $\overline{q_j} = 0$ for all $j \leq n$, so $\xi = 0$.
\end{proof}

\subsection{The counit \ts{\epsilon_M} is an injection}\label{CounitInj}

We continue to write $\V:=\h{U(\C)}$ and also write $V:=\hK{U(\C)}$. 

\begin{prop} Let $f\in A$ be such that $\L\cdot f=\A$ and let $M$ be a finitely generated $U$-module. Then the natural map 
\[\epsilon_M\colon M[f]\otimes_VU\to M\] 
of $U$-modules is injective. In particular $M[f]$ is finitely generated as a $V$-module.
\end{prop}

\begin{proof} Using Proposition \ref{DpUSubmod} we see that the image $M[f]\cdot U$ of $\epsilon_M$ is contained in the $U$-module $M_{\Dp}(f)$. Thus we may view $\epsilon_M$ as a map $M[f]\otimes_V U\to M_{\Dp}(f)$. Using Proposition \ref{DivPow}, choose a $\U$-lattice $\M$ in $M_{\Dp}(f)$ which is stable under all divided powers in $f$. Let $N_1 \subseteq N_2 \subseteq \cdots$ be an ascending chain of $\V_k$-submodules of $\M[f]_k$. Since $\M_k$ is a finitely generated module over the Noetherian ring $\U_k$, its chain $N_1\U_k\subseteq N_2\U_k \subseteq \cdots$ of submodules must stabilize. It now follows from Lemma \ref{CounitModPi}(b) that the chain 
\[ N_1\otimes_{\V_k}\U_k  \subseteq N_2\otimes_{\V_k}\U_k \subseteq \cdots\]
of submodules of $\M[f]_k\otimes_{\V_k}\U_k$ also terminates. But $\U_k$ is a faithfully flat left $\V_k$-module by Lemma \ref{CounitModPi}(a), so the original chain $N_1 \subseteq N_2 \subseteq \cdots $ stops. Therefore $\M[f]_k = \M[f] / \pi \M[f]$ is a finitely generated $\V_k$-module. 

The $\pi$-adic filtration on $\M$ is separated by Nakayama's Lemma, because it is a finitely generated module over the $\pi$-adically complete Noetherian algebra $\U$. So the $\pi$-adic filtration on $\M[f]$ is also separated, and hence  $\M[f]$ is finitely generated over $\V$ by \cite[Chapter I, Theorem 5.7]{LVO}. So $M[f] = \M[f].K$ is finitely generated over $V = \V_K$.

The multiplication map $V \otimes_{\V} \U \to U$ is bijective, so there is an isomorphism of $\U$-modules
\[ M[f] \otimes_{\V} \U \cong M[f]\otimes_V (V \otimes_{\V} \U) \cong  M[f]\otimes_V U.\]
Now $\U$ is a flat $\V$-module by Lemma \ref{CounitModPi}(a) and \cite[Proposition 1.2]{ST}, which implies that $\M[f]\otimes_{\V}\U$ embeds into $M[f]\otimes_{\V}\U \cong M[f] \otimes_V U$. We saw above that $\M[f]$ is finitely generated over $\V$, so $\M[f]\otimes_{\V}\U$ is a $\U$-lattice in $M[f] \otimes_V U$. Now 
\[ (\M[f]\otimes_{\V}\U)_k \cong \M[f]_k \otimes_{\V_k} \U_k\]
embeds into $\M_k$ by Lemma \ref{CounitModPi}(b). This means that the associated graded of $\epsilon_M$ (with respect to the $\pi$-adic filtrations on $M[f]\otimes_V U$ and $M_{\Dp}(f)$ determined by the $\U$-lattices $\M[f] \otimes_{\V} \U$ and $\M$ respectively) is injective. Therefore $\epsilon_M$ is also injective by \cite[Chapter I, Theorem 4.2.4(5)]{LVO}.

That $M[f]$ is finitely generated follows from the facts that $U$ and $V$ are Noetherian and that $U$ is a faithfully flat left $V$-module by Corollary \ref{Lieflat}. 
\end{proof}

\begin{cor} Let $F=\{f_1,\ldots,f_r\}\subset A$ be such that $\L\cdot (f_1,\ldots, f_r)=\A^r$ and let $\C := C_{\L}(F)$. Then for any finitely generated $U$-module $M$, the natural map \[ M[F]\otimes_{\hK{U(\C)}}U\to M \] is an injection and $M[F]$ is finitely generated as a $\hK{U(\C)}$-module.
\end{cor}
\begin{proof}  We proceed by induction on $r$, the case $r=1$ being given by the Proposition. Let $V:=\hK{U(C_\L(f_r))}$ and $W:=\hK{U(\C)}$. Since $\L\cdot f_r=\A$, the Proposition implies that $M[f_r]$ is finitely generated as a $V$-module and $M[f_r]\otimes_V U\to M$ is injective. However, $C_{\L}(f_r)$ is a smooth $(\R,\A)$-Lie algebra by Lemma \ref{Cent} and $C_{\L}(f_r)\cdot (f_1,\ldots,f_{r-1})=\A^{r-1}$ because $\L \cdot (f_1,\ldots,f_r) = \A^r$, so the induction hypothesis implies that $M[F] = M[f_r][\{f_1,\ldots,f_{r-1}\}]$ is finitely generated as a $W$-module and $M[F]\otimes_W V\to M[f_r]$ is injective. 

Since $U$ is a flat $V$-module by Corollary \ref{Lieflat}, $(M[F] \otimes_W V) \otimes_V U \to M[f_r]\otimes_VU$ is also injective, and the result follows by the associativity of tensor product.
\end{proof}
 
It follows immediately from the Corollary that the adjoint pair of functors in Proposition \ref{MdpF} restricts to an adjoint pair of functors between finitely generated $\hK{U(\L)}$-modules and finitely generated $\hK{U(C_\L(F))}/(F)$-modules, whenever $\L \cdot (f_1,\ldots,f_r) = \A^r$.
 
\subsection{Constructing maps from \ts{M_{\Dp}(f)} to \ts{M[f]}}\label{EeJay}

Suppose again that $f\in A$ and $x\in \L$ are such that $x\cdot f=1$. For each $j\ge 0$ and $v \in M_{\Dp}(f)$, the infinite series
\[ e_j(v) := \sum_{n = j}^\infty v\frac{f^n}{n!}\binom{n}{j}(-x)^{n-j} \]
converges to an element of $M$ because $vf^n/n! \to 0$ and because $\binom{n}{j} (-x)^{n-j} \in \U$ for all $n \geq j$. Thus we have defined an infinite collection of functions\[e_j\colon M_{\Dp}(f)\to M, \quad\quad j \geq 0.\]
\begin{lem}  \be
\item $e_j$ is continuous.
\item $e_j(M_{\Dp}(f))\subset M[f]$.
\item If $v\in M_{\Dp}(f)$ then $e_j(v) \to 0$ as $j \to \infty$.
\ee\end{lem}
\begin{proof}(a) By Proposition \ref{DivPow} we can find a $\U$-lattice $\M$ in $M$ which is stable under the action of all $f^i/i!$. We may view $\M$ as the unit ball with respect to a Banach norm on $M$. Then $M_{\Dp}(f)\cap \M$ is the unit ball in the closed subspace $M_{\Dp}(f)$. We see by examining the definition of $e_j$ that $e_j(\M\cap M_{\Dp}(f))\subset \M$ and so $e_j$ is continuous.

(b) Suppose $v\in M_{\Dp}(f)$. We can rewrite $e_j(v)$ as $e_j(v)= \sum_{i=0}^\infty v'\frac{f^i(-x)^i}{i!}$ where $v' := vf^j/j!$. Now 
\[ \frac{f^i(-x)^i}{i!}f = \frac{f^{i+1}(-x)^{i}}{i!} - \frac{f^i(-x)^{i-1}}{(i-1)!} \qmb{for} i \geq 1\]
so the expression for $e_j(v)f$ telescopes to give zero:
\[e_j(v)f = v'f + \sum_{i=1}^\infty v'\left(\frac{f^{i+1}(-x)^i}{i!}- \frac{f^i(-x)^{i-1}}{(i-1)!} \right) = 0.\]
(c) This is clear from the defining formula for $e_j(v)$, because $v\frac{f^n}{n!} \to 0$ as $n \to \infty$ and because $ \binom{n}{j} x^{n-j}$ preserves the $\U$-lattice $\M$ for all $n$ and $j$.\end{proof}

\begin{cor} Let $G \subset A$ be such that $x \cdot G = 0$. Then
\[e_j\left(M_{\Dp}(\{f\} \cup G)\right)\subset M[f]_{\Dp}(G) \qmb{for all} j \geq 0.\]
\end{cor}
\begin{proof}  Part (b) of the Lemma gives the inclusion $e_j\left(M_{\Dp}(\{f\} \cup G)\right) \subset M[f]$. Let $v\in M_{\Dp}(\{f\} \cup G)$ and $g \in G$. Then since $g$ commutes with $x$ and $f$ inside $U$, \[e_j(v)\frac{g^n}{n!}=e_j\left(v\frac{g^n}{n!}\right)\to 0\] as $n\to \infty$ by part (a) of the Lemma.
\end{proof}

\subsection{The counit \ts{\epsilon_M} has image \ts{M_{\Dp}(F)}}\label{CounitIm}
In this Subsection, we will show that $M_{\Dp}(F)$ is generated as a $U$-module by $M[F]$. The heart of the proof of this statement is contained in the following technical
\begin{lem} Let $f \in A$ and $x \in \L$ be such that $x \cdot f = 1$. Suppose that $G \subset A$ is such that $x \cdot G = 0$ and let $F = \{f\} \cup G$. Then
\[ M_{\Dp}(F) \subset M[f]_{\Dp}(G) \cdot U\]
for every finitely generated $U$-module $M$.
\end{lem}
\begin{proof} Let $\C = C_{\L}(f)$ and $V=\hK{U(\C)}$, and define
\[ N := M[f]_{\Dp}(G) \cdot V \subset M[f].\] 
Since $V$ is Noetherian and $M[f]$ is finitely generated over $V$ by Proposition \ref{CounitInj}, we can find a finite generating set $w_1,\ldots, w_m$ in $M[f]_{\Dp}(G)$ for $N$. 

Define $e_j\colon M_{\Dp}(f)\to M$ as in Subsection \ref{EeJay}. Given $v\in M_{\Dp}(F)$, every $e_j(v)$ lies in $M[f]_{\Dp}(G)$ by Corollary \ref{EeJay}, so we can choose $v_{ij} \in V$ such that \[ e_j(v) = \sum_{j=1}^m w_iv_{ij} \qmb{for all} j \geq 0.\]
Because $e_j(v) \to 0$ as $j \to \infty$ by Lemma \ref{EeJay}(c) and because the topology on $N$ can be defined by the $\V$-lattice $\sum_{j=1}^m  w_j\V$, we may assume that $\lim\limits_{j\to\infty} v_{ij} = 0$ for each $i$. Therefore the series $\sum_{j=0}^\infty  v_{ij}x^j$ converges to an element $z_i \in U$ for each $i = 1, \ldots, m$. Now
\begin{eqnarray*} \sum_{i=1}^m w_i z_i &=& \sum_{i=1}^m \sum_{j=0}^\infty w_i v_{ij} x^j = \sum_{j=0}^\infty  e_j(v)x^j =\\
&=& \sum_{j=0}^\infty  \sum_{n=j}^\infty v\frac{f^n}{n!} \binom{n}{j} (-x)^{n-j} x^j = \\
&=& \sum_{n=0}^\infty \left(\sum_{j=0}^n (-1)^j \binom{n}{j} \right)  v\frac{f^n}{n!}(-x)^n = \\
&=& \sum_{n=0}^\infty (1 - 1)^n  v\frac{f^n}{n!}(-x)^n = v,\end{eqnarray*}
so $M_{\Dp}(F)$ is contained in $M[f]_{\Dp}(G)\cdot U$ as required.
\end{proof}

\begin{thm} Let $F=\{f_1,\ldots,f_r\}\subset A$ and suppose that $\L\cdot (f_1,\ldots,f_r)=\A^r$. Then
\[ M[F] \cdot U = M_{\Dp}(F)\]
for every finitely generated $U$-module $M$.
\end{thm}
\begin{proof} The forward inclusion follows from Proposition \ref{DpUSubmod}. To prove the reverse inclusion we proceed by induction on $r$, the base case $r=1$ being given by the Lemma with $G = \emptyset$.

Suppose that $r>1$, write $f := f_r$ and $G := \{f_1,\ldots,f_{r-1}\}$. By the assumption on $\L$ we can find $x \in \L$ such that $x \cdot f = 1$ and $x \cdot G = 0$. Then 
\[M_{\Dp}(F) \subset M[f]_{\Dp}(G) \cdot U \]
by the Lemma. Let $\C = C_\L(f)$; then $\C \cdot (f_1,\ldots, f_{r-1}) = \A^{r-1}$ because $\L \cdot (f_1,\ldots,f_r)=  \A^r$, and $\C$ is a smooth $(\R,\A)$-Lie algebra by Lemma \ref{Cent}. Moreover $M[f]$ is a finitely generated $V := \hK{U(\C)}$-module by Corollary \ref{CounitInj}, so 
\[M[f]_{\Dp}(G) \subset M[f][G] \cdot V  \]
by the induction hypothesis. Therefore
\[ M_{\Dp}[F] \subset M[f]_{\Dp}(G) \cdot U \subset (M[f][G] \cdot V) \cdot U = M[F] \cdot U\]
because $M[f][G] = M[F]$.
\end{proof}

\subsection{Kashiwara's equivalence for \ts{\hK{U(\L)}}}\label{AffKash} 
The following elementary result will be useful on more than one occasion; we could not locate a reference to it in the literature.

\begin{prop} Let $\C$ and $\D$ be two categories, let $T : \D \to \C$ be a functor and let $S : \C \to \D$ be a left adjoint to $T$. Suppose that
\be \item the counit morphism $\epsilon : ST\to 1_{\D}$ is an isomorphism, and
\item $S$ reflects isomorphisms.
\ee Then $S$ and $T$ are mutually inverse equivalences of categories.
\end{prop}\begin{proof} Let $\eta : 1_{\C} \to TS$ be the unit morphism. It will be sufficient to show that $\eta_N : N \to TSN$ is an isomorphism for any $N \in \C$. Now $\epsilon_{SN} \circ S(\eta_N) = 1_{SN}$ by a counit-unit equation and $\epsilon_{SN}$ is an isomorphism by (a), so $S(\eta_N)$ is also an isomorphism. Hence $\eta_N$ is an isomorphism by (b).
\end{proof}

Here is the main result of Section \ref{ConthKUL}.
\begin{thm} Let $F=\{f_1,\ldots,f_r\}$ be a subset of $A$ such that $\L\cdot(f_1,\ldots,f_r)=\A^r$ and write $\C=C_\L(F)$. The functors 
\[ N \mapsto N \otimes_{\hK{U(\C)}}\hK{U(\L)} \qmb{and} M \mapsto M[F]\] 
are mutually inverse equivalences of categories between the category of finitely generated $\hK{U(\C)}/(F)$-modules and the category of finitely generated $\hK{U(\L)}$-modules $M$ such that $M=M_{\Dp}(F)$.
\end{thm}
    
\begin{proof} Write $V=\hK{U(\C)}$ and let $S=(-)[F]$ denote the functor from finitely generated $U$-modules to finitely generated $V/(F)$-modules given by Corollary \ref{CounitInj}, and let $T:= -\otimes_VU$ denote its left adjoint. If $N$ is a finitely generated $V/(F)$-module then $(TN)_{\Dp}(F)\supset STN\supset N\otimes 1$. Now $(TN)_{\Dp}(F)$ is a $U$-submodule of $TN$ by Proposition \ref{DpUSubmod}, and it contains $N \otimes 1$ which generates $TN = N \otimes_V U$ as a $U$-module, so $TN=(TN)_{\Dp}(F)$.

If $M$ is a finitely generated $U$-module such that $M=M_{\Dp}(F)$, then the counit morphism $\epsilon_M\colon TSM\to M$ is an isomorphism by Corollary \ref{CounitInj} and Theorem \ref{CounitIm}. Since $U$ is a faithfully flat left $V$-module by Corollary \ref{Lieflat}, the functor $S$ reflects isomorphisms. The result now follows from the Proposition.
\end{proof}

\section{Kashiwara's equivalence for \ts{\w{U(L)}}} \label{IplusInat}

\subsection{Normalisers in \ts{L}}\label{Norm}

Suppose that $I$ is an ideal in a commutative $R$-algebra $A$ and $L$ is a $(R,A)$-Lie algebra.

\begin{defn} The \emph{normaliser of $I$ in $L$}, $N_L(I):=\{x\in L\st \rho(x)(I)\subset I \}.$
\end{defn}

We will use the abbreviation $N := N_L(I)$ in this subsection.

\begin{lem} \hfill \be \item $N$ is an $(R,A)$-Lie subalgebra of $L$. \item $IL$ is an $(R,A)$-Lie ideal in $N$. \item $N/IL$ is naturally an $(R,A/I)$-Lie algebra. \item $U(L)/IU(L)$ is naturally a $U(N/IL)-U(L)$-bimodule.
\ee
\end{lem}

\begin{proof} (a) It suffices to show that $N$ is an $A$-linear submodule of $L$ and a sub-$R$-Lie algebra. Both facts are immediate from the definitions.

(b) Clearly $IL$ is $A$-submodule of $N$. Suppose $n\in N$, $a\in I$ and $x\in L$. Then $[n,ax]=a[n,x]+\rho(n)(a) x\in IL$, so $[L, IL] \subseteq IL$.

(c) Certainly $N/IL$ is naturally an $A/I$-module and an $R$-Lie algebra. Let the anchor map $\rho_{N/IL}\colon N/IL\to \Der_R(A/I)$ be given by \[\rho_{N/IL}(x+IL)(a+I)=\rho(x)(a)+I\] which is well-defined by the definition of $N$. The verification that this satisfies the conditions for an anchor map is routine.

(d) Consider the natural left-action of $N$ on $U(L)$ given by restricting the action of $L$. If $n\in N$, $a\in I$ and $u\in U(L)$, then we can compute \[n(au)=\rho(n)(a)u+a(nu)\in IU(L) \] by considering the definition of $N$. Thus the image of $U(N)$ in $U(L)$ is contained in the idealizer subring of $IU(L)$. Since $IL$ acts trivially on $U(L)/IU(L)$, using the universal property of enveloping algebras we see that the left-action of $U(N)$ on $U(L)$ descends to a left-action of $U(N/IL)$ on $U(L)/IU(L)$ that commutes with the natural right-action of $U(L)$.  
\end{proof}

\subsection{Standard bases}\label{Basis} Suppose that $I$ is an ideal in a commutative $R$-algebra $A$ and $L$ is a $(R,A)$-Lie algebra.

\begin{defn} We say that a subset $\{x_1,\ldots, x_d\}$ of $L$ is an \emph{$I$-standard basis} if
\be\item $\{x_1,\ldots, x_d\}$ is an $A$-module basis for $L$,
\item there is a generating set $\{f_1,\ldots, f_r\}$ for $I$ with $r \leq d$ such that
\item $x_i \cdot f_j = \delta_{ij}$ for all $1\leq i \leq d$ and $1 \leq j \leq r$.
\ee\end{defn}

Under the assumption that $L$ has an $I$-standard basis, we can explicitly compute the normaliser $N_L(I)$ as follows.

\begin{prop} Suppose that $\{x_1,\ldots,x_d\}$ is an $I$-standard basis for $L$ with corresponding generating set $F = \{f_1,\ldots, f_r\}$ for $I$. Then
\be \item  $C := C_L(F) = Ax_{r+1} \oplus \cdots \oplus Ax_d$,
\item  $N := N_L(I) = I x_1 \oplus \cdots \oplus I x_r \oplus C$,
\item $N / IL = (A/I)\overline{x_{r+1}} \oplus \cdots \oplus (A/I)\overline{x_d}$,
\item $N / IL \cong C / IC$ as $(R,A/I)$-Lie algebras.
\ee
\end{prop}
\begin{proof} This is routine.\end{proof}

\noindent \textbf{Until the end of Section \ref{IplusInat}, we assume that :}
\begin{itemize}
\item $I$ is a radical ideal in the $K$-affinoid algebra $A$,
\item $\{x_1,\ldots, x_n\}$ is an $I$-standard basis for the $(K,A)$-Lie algebra $L$.
\end{itemize}
We also fix an affine formal model $\A$ in $A$. Because $(\pi^n x_i) \cdot (f_j / \pi^n) = \delta_{ij}$ for all $i,j$, we see that $\{\pi^n x_1,\ldots, \pi^n x_d\}$ is again an $I$-standard basis for $L$ for any integer $n$. So by replacing $\{x_1,\ldots,x_d\}$ by a $\pi$-power multiple if necessary and applying \cite[Lemma 6.1(c)]{DCapOne}, we will assume that
\[\L := \A x_1 + \ldots + \A x_d\]
is a free $\A$-Lie lattice in $L$. We also fix a generating set $F := \{f_1,\ldots,f_r\}$ for $I$ such that $x_i\cdot f_j = \delta_{ij}$ for all $i,j$.

\subsection{The transfer-bimodule}\label{TransBi} 
With the notation established above, we have a closed embedding of affinoid varieties 
\[ Y := \Sp(A/I) \hookrightarrow X:= \Sp(A)\]
defined by the ideal $I$. We call the $(K,A/I)$-Lie algebra
\[L_Y:=\frac{N_L(I)}{IL}\]
 the \emph{pullback} of $L$ to $Y$. Since $L$ and $L_Y$ are free modules of finite rank by Proposition \ref{Basis}, it follows from Theorem \ref{RevComp} that $\w{U(L)}$ and $\w{U(L_Y)}$ are Fr\'echet-Stein algebras.  We will use this basic fact without further mention in what follows.
\begin{defn}  The \emph{transfer-bimodule} $\w{U(L)}_{Y\to X}$ is defined by \[\w{U(L)}_{Y\to X}:=A/I\otimes_A \w{U(L)}\cong \w{U(L)}/I\w{U(L)}.\]
\end{defn}

\begin{prop} $\w{U(L)}_{Y\to X}$ is a $\w{U(L)}$-coadmissible $(\w{U(L_Y)},\w{U(L)})$-bimodule.
\end{prop}
\begin{proof} $\w{U(L)}_{Y\to X}$ is finitely presented as a right $\w{U(L)}$-module, so it is coadmissible by \cite[Corollary 3.4]{ST}. Let  $\I := I \cap A$ and $\N:=\frac{N_\L(\I)}{\I\L}$. Then $\N$ is an $(\R,\A/\I)$-Lie algebra by Lemma \ref{Norm}(c). Now $U(\L)/\I U(\L)$ is  a $U(\N)$ -- $U(\L)$-bimodule by Lemma \ref{Norm}(d), so $\hK{U(\L)}/I\hK{U(\L)}$ is a $\hK{U(\N)}$ -- $\hK{U(\L)}$-bimodule. 

Since $\L$ is a flat $\R$-module, multiplication by $\pi^n$ induces an isomorphism 
\[\frac{N_{\L}(\I)}{\I\L} \stackrel{\cong}{\longrightarrow} \frac{N_{\pi^n \L}(\I)}{ \I(\pi^n \L)}.\] 
Hence $\hK{U(\pi^n \L)}/I\hK{U(\pi^n \L)}$ is a $\hK{U(\pi^n \N)}$ -- $\hK{U(\pi^n \L)}$-bimodule and the  homomorphism $\hK{U(\pi^n\N)}\to \End_{\hK{U(\pi^n \L)}}(\hK{U(\pi^n \L)}/I\hK{U(\pi^n \L)})^{\op}$ is continuous.

Finally, $N_{\L}(\I) \cap IL = \I \L$ because $\L \cap IL = \I \L$ as $\L$ is a free $\A$-module.  Hence $\N$ embeds naturally into $N_L(I) / IL = L_Y$ and its image is an $\A/\I$-Lie lattice in $L_Y$. Hence $\w{U(L_Y)} \cong \invlim \hK{U(\pi^n \N)}$ by Definition \ref{RevComp}, and therefore 
\[\w{U(L)}_{Y\to X} \cong \invlim \hK{U(\pi^n \L)}/I\hK{U(\pi^n \L)}\]
is a $\w{U(L)}$-coadmissible $(\w{U(L_Y)},\w{U(L)})$-bimodule by \cite[Definition 7.3]{DCapOne}.
\end{proof}

\subsection{The functors \ts{\iota_+} and \ts{\iota^\natural}}\label{FrechKash} By Proposition \ref{TransBi}, the bimodule $\w{U(L)}_{Y\to X}$ satisfies the conditions for the right-module version of Lemma \ref{Revwotimes}. Thus 
\[N   \underset{\w{U(L_Y)}}{\w{\otimes}} \w{U(L)}_{Y\to X}\]
is a coadmissible $\w{U(L)}$-module whenever $N$ is a coadmissible $\w{U(L_Y)}$-module.
\begin{defn} Let $M$ be a coadmissible $\w{U(L)}$-module and let $N$ be a coadmissible $\w{U(L_Y)}$-module.
\be
\item The \emph{pushforward} of $N$ to $X$ is the coadmissible $\w{U(L)}$-module \[ \iota_+N:= N\underset{\w{U(L_Y)}}{\w{\otimes}} \w{U(L)}_{Y\to X}. \]
\item The \emph{pullback} of $M$ to $Y$ is the $\w{U(L_Y)}$-module 
\[ \iota^\natural M:=\Hom_{\w{U(L)}}\left(\w{U(L)}_{Y\to X},M\right). \]
\ee\end{defn} 

Note the right action of $\w{U(L_Y)}$ on $\iota^\natural M$ comes from its left action on $\w{U(L)}_{Y\to X}$ given in Proposition \ref{TransBi}. That $\iota^\natural$ defines a functor from $\w{U(L)}$-modules to $\w{U(L_Y)}$-modules is clear, and the universal property of $\w\otimes$ shows that $\iota_+$ is also a functor. 

\subsection{Vectors where \ts{S \subset A} acts topologically nilpotently}\label{Minfty}

Before we can prove the main theorem of Section \ref{IplusInat} we will need some more definitions.

\begin{defn} Let $M$ be a coadmissible $U$-module and let $S$ be a subset of $A$. Define
\[M_\infty(S):=\{ m\in M \st ms^k\to 0\mbox{ as }k\to \infty\mbox{ for all }s\in S\}.\]
We say that $S$ acts \emph{topologically nilpotently} on $m \in M$ precisely when $m \in M_\infty(S)$.  We say that $S$ acts \emph{locally topologically nilpotently} on $M$ if $M = M_\infty(S)$.
\end{defn}

We begin our study of $M_\infty(S)$ by doing some analysis. We will write
\[ U := \w{U(L)} \qmb{and} U_n := \hK{U(\pi^n \L)} \qmb{ for any } n\geq 0.\]
If $M$ is a coadmissible $U$-module, then $M_n$ will always mean the finitely generated $U_n$-module $M \otimes_U U_n$.

\begin{lem} Let $M$ be a coadmissible $U$-module. Then 
\be \item $M_\infty(f) \subset M_\infty(af)$ for any $f \in A$ and $a \in \A$, and 
\item $M_\infty(Kf) \cap M_\infty(Kg) \subset M_\infty(K(f+g))$ for any $f,g \in A$. \ee
\end{lem}
\begin{proof} Note that $M$ is the inverse limit of the $K$-Banach spaces $M_n$, so a sequence in $M$ converges to an element $z$ of $M$ if and only if the image of this sequence in each $M_n$ converges to the image of $z$ in $M_n$. Fix $n \geq 0$ and a $U_n$-lattice $\M_n$ in $M_n$.

(a) Since $a \in \A$ preserves $\M_n$, $mf^k \to 0$ in $M_n$ forces $m f^k a^k \to 0$ in $M_n$.

(b) Suppose that $m \in M_\infty(Kf) \cap M_\infty(Kg)$; it will be enough to show that $m(f+g)^k \to 0$ as $k\to \infty$ in $M_n$. Choose $r \geq 0$ such that $\pi^r f$ and $\pi^r g$ both lie in $\A$. Now since $m \in M_\infty(f/\pi^r) \cap M_\infty(g/\pi^r)$, 
\[ m (f/\pi^r)^k \to 0 \qmb{and} m (g/\pi^r)^k \to 0\qmb{as} k\to \infty\]
in $M_n$ so for every $s \geq 0$ there exists an integer $T\geq 0$ such that 
\[ m(f/\pi^r)^k \in \pi^s \M_n \qmb{and} m(g/\pi^r)^k \in \pi^s \M_n \qmb{whenever} k \geq T.\]
Now suppose that $i, j\geq 0$ are such that $i + j \geq 2T$. Then either $i \leq j$ in which case $j \geq T$ and 
\[ m f^i g^j = m (g/\pi^r)^j (\pi^r f)^i \pi^{r(j-i)} \in \pi^s \M_n,\]
or $i \geq j$ in which case $i \geq T$ and 
\[ m f^i g^j = m (f/\pi^r)^i (\pi^r g)^j \pi^{r(i-j)} \in \pi^s \M_n,\]
because $(\pi^rf)^i$ and $(\pi^rg)^j$ preserve $\pi^s \M_n$ by construction. Therefore for all $s \geq 0$ there exists $T \geq 0$ such that
\[ m (f+g)^k = \sum_{i+j=k} m f^i g^j \binom{k}{i} \in \pi^s \M_n\]
whenever $k \geq 2T$, and hence $m(f+g)^k \to 0$ in $M_n$ as $k\to \infty$ as required. \end{proof}

It is interesting to note here that the fact that $f$ and $g$ commute is crucial, since for example the sum of the two non-commuting square-zero elements $\begin{pmatrix} 0 & 1 \\ 0 & 0 \end{pmatrix}$ and $\begin{pmatrix} 0 & 0 \\ 1 & 0 \end{pmatrix}$ in the matrix ring $M_2(K)$ is a unit.
\begin{cor} Let $M$ be a coadmissible $U$-module. Then 
\[M_\infty(W) = M_\infty(A W)\]
for any $K$-vector subspace $W$ of $A$.
\end{cor}
\begin{proof} Note that $AW = \A K W = \A W$, and that the inclusion $M_\infty(AW) \subset M_\infty(W)$ is trivial.  Let $m\in M_\infty(W)$ and suppose that $w_1,\ldots,w_t \in W$ and $a_1,\ldots,a_t \in \A$. Then $m \in M_\infty(K w_i)$ for each $i$ and hence $m \in M_\infty(K a_i w_i)$ for each $i$ by part (a) of the Lemma. Hence $m \in M_\infty\left( K \cdot \sum_{i=1}^t a_iw_i \right)$ by part (b) of the Lemma, so $m \in M_\infty(w)$ for all $w\in AW$.
\end{proof}

\subsection{Proposition}\label{MinftyUsub} Let $M$ be a coadmissible $U$-module. Then 
\[ M_{\infty}(KS)\cong \invlim (M_n)_{\Dp}(S/\pi^n)\]
for any subset $S$ of $A$.
\begin{proof} Let $m_n$ denote the image of $m\in M$ in $M_n$ and recall from Subsection \ref{MdpF} that $m_n$ lies in $N_n := (M_n)_{\Dp}(S/\pi^n)$ if and only if 
\[ \lim\limits_{k\to\infty} m_n(s/\pi^n)^k/k! =0 \qmb{for all} s \in S.\]
Since the map $M_{n+1} \to M_n$ is continuous,  $(N_n)$ forms an inverse system.

Choose an integer $r \geq 0$ such that $\pi^{rk} / k! \to 0 \qmb{as} k\to \infty.$ Now if $m\in M_\infty(KS)$ and $s \in S$ then $\lim\limits_{k\to\infty} m (s / \pi^{n+r})^k \to 0$, so 
\[ \lim\limits_{k\to\infty}m \frac{(s/\pi^n)^k}{k!} = \lim\limits_{k\to\infty}m \left(\frac{s}{\pi^{n+r}}\right)^k \cdot \frac{\pi^{rk}}{k!} = 0\]
in $M$. Thus $m_n \in N_n$ for each $n \geq 0$ because the map $M \to M_n$ is continuous. Conversely, suppose that $m_n \in N_n$ for all $n$. Let $s \in S$ and fix $n, t \geq 0$. Now $m_{n+t} (s / \pi^{n+t})^k/k! \to 0$ in $M_{n+t}$ by assumption and the map $M_{n+t} \to M_n$ is continuous, so 
\[ \lim\limits_{k\to\infty} m_n(s / \pi^t)^k = \lim\limits_{k\to\infty} m_n \frac{ (s / \pi^{n+t})^k }{k!} \cdot \pi^{nk}k! =0 \qmb{for all} n\geq 0.\]
Hence $\lim\limits_{k\to\infty} m (s / \pi^t)^k \to 0$ in $M$ for all $t \geq 0$ and $m \in M_\infty(KS)$.
\end{proof}

\begin{cor} $M_{\infty}(I)$ is a coadmissible $U$-module whenever $M$ is a coadmissible $U$-module.
\end{cor}
\begin{proof} Recall from Subsection \ref{Basis} that $F = \{f_1,\ldots,f_r\}$ generates the ideal $I$ of $A$. Note that for any $n \geq 0$, $\pi^n \L$ is again a free $\A$-Lie lattice in $L$ such that 
\[ (\pi^n \L) \cdot \left(\frac{f_1}{\pi^n}, \cdots, \frac{f_r}{\pi^n}\right) = \A^r,\]
so $(M_n)_{\Dp}(F/\pi^n)$ is a closed $U_n$-submodule of $M_n = M \otimes_U U_n$ by Proposition \ref{DpUSubmod}. By Corollary \ref{Minfty} and Proposition \ref{MinftyUsub},
\[ M_\infty(I) = M_\infty(A \cdot KF) = M_\infty(KF) \cong \invlim (M_n)_{\Dp}(F/\pi^n).\]
Hence $M_\infty(I)$ is a closed $U$-submodule of $M$. Now apply \cite[Lemma 3.6]{ST}.
\end{proof} 

\subsection{The module \ts{M[F]} is coadmissible} \label{MFcoadm}
It follows from Lemma \ref{MdpF} that $\C := C_{\L}(F)$ is a smooth $\A$-Lie lattice in $C := C_L(F)$, so the Banach algebras $V_n := \hK{U(\pi^n \C)}$ give a presentation $\invlim V_n$ of $V := \w{U(C)}$ as a Fr\'echet-Stein algebra by Theorem \ref{RevComp}. There is a natural continuous map of Fr\'echet-Stein algebras
\[ V = \w{U(C)} \longrightarrow U = \w{U(L)}.\]

\begin{thm} Let $M$ be a coadmissible $U$-module. Then $M[F]$ is a coadmissible $V$-module.
\end{thm}
\begin{proof} By Corollary \ref{MinftyUsub}, $M_\infty(I)$ is a coadmissible $U$-submodule of $M$. Since 
\[ M[F] = M_\infty(I)[F]\]
we may assume that $M = M_\infty(I)$. Let $M_n := M \otimes_U U_n$ and $F_n := F/\pi^n$, and note that $M_n[F_n] = M_n[F]$ for all $n \geq 0$. Now the image of $M$ generates $M_n$ as a $U_n$-module, and this image is contained in $N_n := (M_n)_{\Dp}(F_n)$ by Proposition \ref{MinftyUsub} since $M = M_\infty(AF)$.  It follows from Proposition \ref{DpUSubmod} that $N_n = M_n$, and therefore the counit morphism 
\[ \epsilon_n : M_n[F_n] \otimes_{V_n} U_n \longrightarrow M_n\]
is an isomorphism for all $n \geq 0$ by Theorem \ref{AffKash}.	

The $U_{n+1}$-linear map $M_{n+1}\to M_n$ induces a $V_{n+1}$-linear map $M_{n+1}[F_{n+1}]\to M_n[F_n]$ and therefore a $V_n$-linear map
\[ \varphi_n : M_{n+1}[F_{n+1}] \otimes_{V_{n+1}} V_n \longrightarrow M_n[F_n]\]
which features in the following commutative diagram:
\[ \xymatrix{ (M_{n+1}[F_{n+1}] \underset{V_{n+1}}{\otimes}V_n) \underset{V_n}{\otimes} U_n \ar^{\cong}[r]\ar_{\varphi_n \otimes 1}[dd]  &   (M_{n+1}[F_{n+1}] \underset{V_{n+1}}{\otimes} U_{n+1}) \underset{U_{n+1}}{\otimes} U_n \ar^{\epsilon_{n+1} \otimes 1}[d]  \\    &  M_{n+1} \underset{U_{n+1}}{\otimes}U_n \ar[d]^{\alpha_n} \\   M_n[F_n] \underset{V_n}{\otimes} U_n \ar[r]_{\epsilon_n} & M_n.}\]
The map $\alpha_n$ is an isomorphism because $M$ is a coadmissible $U$-module, so $\varphi_n \otimes 1$ is also an isomorphism since $\epsilon_n$ and $\epsilon_{n+1}$ are isomorphisms. But $U_n$ is a faithfully flat $V_n$-module by Corollary \ref{Lieflat}, so $\varphi_n$ is also an isomorphism.  Since each $M_n[F_n]$ is a finitely generated $V_n$-module by Corollary \ref{CounitInj},  $(M_n[F_n])_n$ is a coherent sheaf for $V_\bullet$, so $\invlim M_n[F_n]$ is a coadmissible $V$-module. 

Finally, since $M_n[F_n] = M_n[F]$, the sequence of $V_n$-modules 
\[ 0 \to M_n[F_n] \to M_n \stackrel{\psi_n}{\longrightarrow} M_n^r\]
is exact, where $\psi_n(m) = (mf_1,\ldots,mf_r)$ for all $m\in M_n$. Hence the sequence 
\[ 0 \to \invlim M_n[F_n] \to M \stackrel{\psi}{\longrightarrow} M^r\]
of $V$-modules is also exact, where $\psi(m) = (mf_1,\ldots, mf_r)$ for all $m\in M$. Hence $M[F] \cong \invlim M_n[F_n]$ is a coadmissible $V$-module as required.
\end{proof}

\begin{cor} Let $M$ be a coadmissible $U$-module such that $M = M_\infty(I)$. Then the natural map $\beta_n : M[F] \otimes_V V_n \to M_n[F]$ is an isomorphism of $V_n$-modules for all $n \geq 0$.
\end{cor}
\begin{proof} The maps $\beta_n$ assemble to give a morphism $\beta :  (M[F] \otimes_V V_n)_n \to (M_n[F])_n$ in $\Coh(V_\bullet)$, and $\Gamma(\beta) : M[F] \to \invlim M_n[F]$ is an isomorphism by the proof of the Theorem. Therefore $\beta$ is an isomorphism in $\Coh(V_\bullet)$ by \cite[Corollary 3.3]{ST}.
\end{proof}

\subsection{The algebra \ts{\w{U(L_Y)}} }\label{VoverFV}
We are now close to the proof of Theorem \ref{KashAffinoids}. 

\begin{lem} There is a natural isomorphism of Fr\'echet-Stein algebras $V / FV \cong \w{U(L_Y)}$.
\end{lem}
\begin{proof} Let $\C = C_{\L}(F)$, let $g_1,\ldots, g_s$ generate the ideal $I\cap \A$ in $\A$ and let $\B := \A / I \cap\A$. The exact sequence $\C^s \to \C \to \B \otimes_{\A}\C \to 0$ of $\A$-modules induces a complex of $U(\C)$-modules
\[ U(\C)^s \to U(\C) \to U(\B \otimes_{\A} \C) \to 0.\]
Note that $C := \C \otimes_\R K = C_L(F)$ is a free $(K,A)$-Lie algebra by Proposition \ref{Basis}, so $U(A/IA \otimes_A C) \cong A/IA \otimes_A U(C)$ by \cite[Proposition 2.3]{DCapOne}. Hence this complex becomes exact after inverting $\pi$. Since $U(\C)$ is Noetherian, its cohomology modules are killed by a power of $\pi$, so \cite[\S 3.2.3(ii)]{Berth} implies that the complex
\[ \hK{U(\C)}^s \to \hK{U(\C)} \to \hK{U(\B \otimes_{\A} \C)} \to 0\]
is exact. Recalling that $V_n = \hK{U(\pi^n \C)}$, we see similarly that 
\[ V_n^s \to V_n \to \hK{U( \B \otimes_{\A} \pi^n \C)} \to 0\]
is exact for any $n \geq 0$. Now $C/IC \cong L_Y$ by Proposition \ref{Basis} and the image $\D$ of $\B \otimes_{\A} \C$ in $C/IC$ is a $\B$-Lie lattice, so that
\[ \w{U(L_Y)} \cong \invlim \hK{U(\pi^n \D)}\]
by Definition \ref{RevComp}. Since $\sum_{i=1}^s g_iA = I = FA$, we see that 
\[ \hK{U(\pi^n \D)}  \cong \hK{U( \B \otimes_{\A}\pi^n \C)} \cong V_n / FV_n\] 
by \cite[Lemma 2.5]{DCapOne}. Since $\Gamma$ is exact on $\Coh(V_\bullet)$, the sequence
\[ V^s \to V \to \invlim V_n / FV_n \to 0\]
is exact, and hence $\w{U(L_Y)} \cong \invlim \hK{U(\pi^n \D)} \cong \invlim V_n/FV_n \cong V / FV$. \end{proof}

Here is the main result of Section \ref{IplusInat}.

\subsection{Theorem}\label{KashAffinoids} Let $I$ be an ideal in the $K$-affinoid algebra $A$, and let $L$ be a $(K,A)$-Lie algebra which admits an $I$-standard basis.
\be
\item The functor $\iota^\natural$ preserves coadmissibility.
\item The restriction of $\iota^\natural$ to coadmissible $\w{U(L)}$-modules is right adjoint to $\iota_+$.
\item These functors induce an equivalence of categories between the category of coadmissible $\w{U(L_Y)}$-modules and the category of coadmissible $\w{U(L)}$-modules $M$ such that $I$ acts locally topologically nilpotently on $M$.
\ee
\begin{proof}
(a) Using the notation from Subsection \ref{MFcoadm}, we have 
\[\w{U(L_Y)}_{Y\to X} = U/IU = U/FU.\]
Hence for any $U$-module $M$, the rule $\theta \mapsto \theta(1 + FU)$ gives a natural bijection 
\[\iota^\natural M = \Hom_U( U / FU, M) \stackrel{\cong}{\longrightarrow} M[F]\] 
which sends the $\w{U(L_Y)}$-module structure on $\iota^\natural M$ to the $V/FV$-module structure on $M[F]$ under the identification of $\w{U(L_Y)}$ with $V/FV$ given by Lemma \ref{VoverFV}. Hence $\iota^\natural M$ is a coadmissible $\w{U(L_Y)}$-module by  Theorem \ref{MFcoadm}.

(b) For any coadmissible $U$-module $M$ and coadmissible $V/FV$-module $N$, the restriction map
\[ \Hom_U( N \underset{V/FV}{\w\otimes} U/FU, M) \longrightarrow \Hom_U(N \underset{V/FV}{\otimes} U/FU, M)\]
is a bijection by the universal property of $\w\otimes$, leading to natural isomorphisms
\[ \Hom_U( \iota_+N, M) \cong \Hom_{V/FV}(N, \Hom_U(U/FU,M)) \cong \Hom_{V/FV}(N, \iota^\natural M).\]

(c) Let $M$ be a coadmissible $U$-module such that $I$ acts locally topologically nilpotently on $M$ and let $n\geq 0$. Then $M = M_\infty(I)$ by definition, and $M[I] \otimes_V V_n \cong M_n[I]$ by Corollary \ref{MFcoadm}. Therefore the natural map
\[ M[I] \otimes_V U_n \cong (M[I] \otimes_V V_n) \otimes_{V_n} U_n \cong M_n[I] \otimes_{V_n} U_n \longrightarrow M_n\]
is an isomorphism by Theorem \ref{AffKash}. Since $M[I] \otimes_V U_n \cong M[I] \otimes_{V/FV} U_n / FU_n$, passing to the limit shows that the counit of the adjunction
\[ \iota_+ \iota^\natural M = M[I] \underset{V/FV}{\w\otimes} U/FU \longrightarrow M\]
is an isomorphism.  Next, let $N$ be a coadmissible $V/FV$-module and let $N_n = N \otimes_V V_n$. Then 
\[ \iota_+ N = N \underset{V/FV}{\w\otimes} U/FU = \invlim \left(N_n \underset{V_n/FV_n}{\otimes} U_n/FU_n \right)\cong \invlim N_n \underset{V_n}{\otimes} U_n \cong N \underset{V}{\w\otimes} U,\]
so we may apply Corollary \ref{Minfty} and Proposition \ref{MinftyUsub} to obtain
\[(\iota_+N)_\infty(I) = (\iota_+N)(KF) \cong \invlim (N_n \underset{V_n}{\otimes} U_n)_{\Dp}(F_n).\]
Hence $(\iota_+N)_\infty(I) = \iota_+N$ by Theorem \ref{AffKash}, so $I$ acts locally topologically nilpotently on $\iota_+N$. Finally, Corollary \ref{Lieflat} and the right-module version of \cite[Proposition 7.5(c)]{DCapOne} show that $U$ is a faithfully c-flat right $V$-module. Since $\iota_+N \cong N\underset{V}{\w\otimes} U$ for any coadmissible $V/FV$-module $N$, $\iota_+$ reflects isomorphisms. 

Now apply Proposition \ref{AffKash}.
\end{proof}

\section{Kashiwara's equivalence for \ts{\w{\sU(\sL)}}}\label{Pullpush}
Throughout this section, we assume that $\iota\colon Y\to X$ is a closed embedding of rigid $K$-analytic spaces defined by the radical, coherent $\O_X$-ideal $\I$.

\subsection{Restriction of \ts{\sL_X} to \ts{Y}}\label{LXsmooth}
See Appendix A below for a discussion of the pushforward functor $\iota_\ast$ from abelian sheaves on $Y$ to abelian sheaves on $X$, and the pullback functor $\iota^{-1}$ in the opposite direction. We define the \emph{conormal sheaf} on $Y$ \[ \N_{Y/X}^\ast:= \iota^{-1} (\I/\I^2) \]  and the \emph{normal sheaf} on $Y$ \[ \N_{Y/X}:=\Hom_{\O_Y}(\N_{Y/X}^\ast,\O_Y)\] is its dual. Taking the dual of the second fundamental exact sequence \cite[Proposition 1.2]{BLR3} gives an exact sequence 
of coherent sheaves on $Y$
\[ 0 \to \T_Y \to \iota^\ast \T_X \stackrel{\theta}{\to} \N_{Y/X}.\]
For any Lie algebroid $\sL_X$ on $X$, there is a commutative diagram of coherent sheaves on $Y$ with exact rows \[ \xymatrix{   0\ar[r] &  \ar[r] \ar[d]^{\rho_Y} \sL_Y & \ar[d]^{\iota^\ast\rho} \iota^\ast \sL_X &  \\
                                                                                          \ar[r] 0 & \ar[r]         \T_Y &    \ar[r]_\theta        \iota^\ast\T_X &  \N_{Y/X}  
} \] where $\sL_Y := \ker(\theta \circ \iota^\ast\rho :\iota^\ast \sL_X \to \N_{Y/X})$. We can compute the local sections of these sheaves as follows.

\begin{lem} Let $U$ be an open affinoid subvariety of $X$. Write $I = \I(U)$, $L = \sL_X(U)$ and $V=U \cap Y$. Then there are natural isomorphisms of $\O_Y(V)$-modules
\be \item $\N_{Y/X}^\ast(V) \cong I / I^2$, and
\item $\sL_Y(V) \cong N_L(I) / IL$.
\ee\end{lem}
\begin{proof} (a) Since $\I / \I^2$ is supported on $Y$, it is naturally isomorphic to $\iota_\ast \N_{Y/X}^\ast$ by Theorem \ref{ShSuppClsdSub}, so $\Gamma( U \cap Y, \N_{Y/X}^\ast) \cong \Gamma(U, \I/\I^2)$. Since $U$ is affinoid and $0 \to \I^2 \to \I \to \I / \I^2 \to 0$ is an exact sequence of coherent sheaves on $X$, Kiehl's Theorem \cite[Theorem 9.4.3/3]{BGR} implies that $\Gamma(U, \I/\I^2) \cong \I(U) / \I^2(U)$. Similarly, the natural map $(\I \otimes_\O \I)(U) \to \I^2(U)$ is surjective, so the image of $\I^2(U)$ in $I$ equals $I^2$.

(b) Let $A = \O(U)$ so that $\O_Y(V) \cong A/I$. In view of part (a), there is a commutative diagram of $A/I$-modules 
\[ \xymatrix{ \iota^\ast\sL_X(V) \ar[r]^{(\theta \circ \iota^\ast \rho)(V)}\ar[d]_{\cong} & \N_{Y/X}(V) \ar[d]^{\cong} \\ L/IL \ar[r]_(0.35)\alpha &  \Hom_A(I/I^2,A/I) }\]
where the bottom map $\alpha : L / IL \to \Hom_A(I/I^2,A/I)$ is given by 
\[ \alpha(x + IL)(f + I) = \rho(x)(f) + I \qmb{for all} x \in L \qmb{and} f \in A.\]
Thus the kernel of $(\theta \circ \iota^\ast\rho)(V)$ is isomorphic to $\ker \alpha = N_L(I) / IL$ as claimed.
\end{proof}

\begin{cor} If $\sL_Y$ is locally free, then $\sL_Y$ is a Lie algebroid on $Y$.
\end{cor}
\begin{proof} This follows from Lemma \ref{Norm}(c) and Lemma \ref{LXsmooth}(b).\end{proof}

 {\bf From now on, $\sL = \sL_X$ is a Lie algebroid on $X$.}
 
 \subsection{The existence of local standard bases}\label{StdBasis} Our next result is a rigid analytic version of \cite[Theorem A.5.3]{HTT}, which follows from the well-known fact from algebraic geometry that a smooth subvariety of a smooth variety is locally a complete intersection. First, a preliminary
 
\begin{prop} Let $A$ be an affinoid $K$-algebra with an ideal $I$ such that $I^2 \neq I$. Let $L$ be a $(K,A)$-Lie algebra which is free as an $A$-module of rank $d$. Suppose that $I/I^2$ and $N_L(I)/IL$ are free as $A/I$-modules, and that the map $L/IL \to \Hom_{A/I}(I/I^2, A/I)$ is surjective. Then there exists $g \in A$ such that $A \langle 1/g \rangle \otimes_A L$ has an $A\langle 1/g \rangle \otimes_A I$-standard basis, and $I \cdot A\langle g \rangle = A\langle g\rangle$.
\end{prop}
\begin{proof} Since $L/IL$ is a free $A/I$-module of rank $d$ and the sequence
\[ 0\to N_L(I) / IL \to L / IL \to \Hom_{A/I}(I/I^2,A/I) \to 0\]
is exact by assumption, the ranks of $N_L(I)/IL$ and $\Hom_{A/I}(I/I^2,A/I)$ as free $A/I$-modules add up to $d$. Choose an $A/I$-module basis $\{f_1 + I^2,\ldots,f_r + I^2\}$ for $I / I^2$; then we can find $y_1,\ldots,y_d\in L$ such that $y_i\cdot f_j\in \delta_{ij}+ I$ whenever $1 \leq i,j \leq r$, and such that the images of $y_{r+1},\ldots, y_d \in N_L(I)$ in $N_L(I) / IL$ form an $A/I$-module basis for $N_L(I) / IL$. Now $I / \sum_{i=1}^r Af_i$ is killed by some element $a \in 1 + I$ by Nakayama's Lemma. By construction, the image of $\{y_1,\ldots,y_d\}$ in $L/IL$ generates $L/IL$ as an $A/I$-module, so $L / \sum_{i=1}^d Ay_i$ is killed by some element $b \in 1 + I$, again by Nakayama's Lemma. Finally, the determinant $c$ of the matrix $M := (y_i\cdot f_j)_{r \times r}$ lies in $1 + I$ because $M$ is congruent to the identity matrix modulo $I$.

Let $g := abc/\pi \in K$ and let $B := A \langle 1/g \rangle$. Since $a \in B^\times$, the images of $f_1,\ldots,f_r$ in $B$ generate $I \cdot B$. Let $y_j' := 1 \otimes y_j \in B \otimes_AL$; since $b \in B^\times$, the set $\{y'_1,\ldots,y'_d\}$ in $B \otimes_A L$ generates $B \otimes_A L$ as a $B$-module. Since $c \in B^\times$, the matrix $M$ has an inverse with entries in $B$. Define
\[x_i := \sum_{k=1}^r M^{-1}_{ik}y'_k\in B \otimes_AL \qmb{for} i=1,\ldots, r.\]
Then $x_i \cdot f_j = \sum_{j=1}^r M^{-1}_{ik} y'_k \cdot f_j = (M^{-1} M)_{ij} = \delta_{ij}$ for any $1 \leq i,j \leq r$ by construction. Now, define
\[x_i := y_i' - \sum_{\ell=1}^r (y_i' \cdot f_\ell) x_l \qmb{for} i > r.\]
Then $x_i \cdot f_j = 0$ for all $i > r$ and all $j \leq r$. Clearly $\{x_1,\ldots,x_d\}$ still generates $B \otimes_AL$ as a $B$-module; but $B \otimes_AL$ is a free $B$-module of rank $d$ by assumption, and this forces $\{x_1,\ldots,x_d\}$ to be a $B$-module basis.  Thus $\{x_1,\ldots, x_d\}$ is the required $I\cdot B$-standard basis for $B\otimes_A L$. 

Finally, since $a,b,c \in 1 + I$ we see that $1 - abc \in I$. Hence $I \cdot A \langle g \rangle$ contains $1 - abc = 1 - \pi g$ which is invertible in $A\langle g \rangle$, so $I \cdot A\langle g \rangle = A \langle g \rangle$.
\end{proof}

\begin{thm} Suppose that $\theta \circ \iota^\ast \rho : \iota^\ast \sL \to \N_{Y/X}$ is surjective, and that $\N_{Y/X}$ is locally free. Then there is an admissible affinoid covering $\{X_j\}$ of $X$ such that for all $j$, either $\sL(X_j)$ has an $\I(X_j)$-standard basis or $\I(X_j)^2 = \I(X_j)$.
\end{thm}
\begin{proof} Since the problem is local on $X$ and since $\sL$ is locally free, we may assume that $X$ is affinoid and that $\sL(X)$ is a free $\O(X)$-module of rank $d$ say. Now by assumption, we have a short exact sequence of $\O_Y$-modules
\[ 0 \to \sL_Y \to \iota^\ast \sL \stackrel{\theta\circ\iota^\ast \rho}{\longrightarrow} \N_{Y/X} \to 0\]
with the second and third term locally free. Hence $\N_{Y/X}^\ast$ and $\sL_Y$ are both locally free. Therefore there is a Zariski covering $\{D(h_1),\ldots,D(h_m)\}$ of $Y$ such that $\Gamma(D(h_j), \N_{Y/X}^\ast)$ and $\Gamma(D(h_j), \sL_Y)$ are free $\O(D(h_j))$-modules of finite rank for all $j = 1,\ldots, m$. Choose a preimage $g_j \in \O(X)$ for $h_j \in \O(Y)$; then $\I(X) + \sum_{j=1}^m \O(X) g_j = \O(X)$ so we can find $g_0 \in \I(X)$ such that $\{D(g_0),\ldots, D(g_m)\}$ is a Zariski covering of $X$. By \cite[Corollary 9.1.4/7]{BGR}, this covering has a finite affinoid refinement $\X$. For every $U \in \X$, $\N_{Y/X}^\ast(U \cap Y)$ and $\sL_Y(U\cap Y)$ are free $\O(U \cap Y)$-modules of finite rank by construction.

Let $U\in \X$ be such that $\I(U) \neq \I(U)^2$. Since $\iota^\ast \sL\to \N_{Y/X}$ is a surjective morphism between two coherent sheaves,  Kiehl's Theorem \cite[Theorem 9.4.3/3]{BGR} implies that $\Gamma(U \cap Y, \iota^\ast \sL) \to \Gamma(U \cap Y, \N_{Y/X})$ is surjective. Hence $\sL(U) / \I(U)\sL(U) \to \Hom_{\O(U)/\I(U)}(\I(U)/\I(U)^2, \O(U)/\I(U))$ is surjective by Lemma \ref{LXsmooth}(a) so by the Proposition there is an admissible covering $\{U_1,U_2\}$ of $U$ where $\sL(U_1)$ has an $\I(U_1)$-standard basis, and $U_2 \cap Y$ is empty. Since $\I(U_2) = \O(U_2) = \I(U_2)^2$,
\[ \{U_1 : U \in \X\} \cup \{ U_2 : U\in \X\} \cup \{U\in \X : \I(U) = \I(U)^2\}\]
is an admissible affinoid covering of $X$ with the required properties.\end{proof}
 
\subsection{The basis \ts{\B}}\label{BasisB}
We will henceforth assume that
\begin{itemize}
\item $\theta \circ \iota^\ast \rho : \iota^\ast \sL \to \N_{Y/X}$ is surjective,
\item $\N_{Y/X}$ is locally free,
\item $\B$ is the set of affinoid subdomains $U$ of $X$ such that $\sL(U)$ has a smooth Lie lattice and moreover it has an $\I(U)$-standard basis whenever $\I(U) \neq \I(U)^2$.
\end{itemize}
Note that $\B$ is closed under passing to smaller affinoid subdomains, and that under the first two assumptions Theorem \ref{StdBasis} implies that $X$ has an admissible covering by objects in $\B$. Note also that the $\O_Y$-module $\sL_Y$ is then locally free and is therefore a Lie algebroid on $Y$ by Corollary \ref{LXsmooth}. Regarding the condition $\I(U) = \I(U)^2$ in the definition, we remind the reader of the following elementary

\begin{lem} Let $X$ be a connected affinoid variety. Then $\O(X)$ contains no non-trivial idempotent ideals.
\end{lem}
\begin{proof} Since $\O(X)$ is Noetherian, if $I$ is an idempotent ideal in $\O(X)$ then it is finitely generated and Nakayama's Lemma implies that $I(1-e) = 0$ for some $e \in I$. But then $e(1-e) = 0$ so $e$ is an idempotent in $I$. Since $X$ is connected by assumption, $e = 0$ or $e = 1$. In the first case, $I = I(1 - 0) = 0$ and in the second case, $1 \in I$ so $I = \O(X)$.
\end{proof}

Thus if $U$ is a \emph{connected} affinoid subdomain of $X$ then $\I(U) = \I(U)^2$ if and only if either $U \subset Y$ or $U \cap Y = \emptyset$.

\subsection{The pushforward functor \ts{\iota_+}} We have at our disposal the sheaves $\sU := \w{\sU(\sL)}$ on $X$ and $\sW := \w{\sU(\sL_Y)}$ on $Y$. If $U \in \B$ then $\sW(U \cap Y) = \w{U(\sL_Y(U\cap Y))}$ by Theorem \ref{RevLiealg} and
\[ \sL_Y(U \cap Y) \cong \frac{N_{\sL(U)}\left(\I(U)\right)}{\I(U)\sL(U)}\]
by Lemma \ref{LXsmooth}. Hence Proposition \ref{TransBi} implies that $\frac{ \sU(U) }{ \I(U) \sU(U) }$ is a $\sU(U)$-coadmissible $\sW(U\cap Y) - \sU(U)$-bimodule. For every coadmissible $\sW$-module $\sN$ on $Y$, we can therefore form the coadmissible $\sU(U)$-module
\[ (\iota_+\sN)(U) := \sN(U \cap Y)\hspace{0.2cm}\underset{\sW(U\cap Y)}{\w{\otimes}}\hspace{0.2cm} \frac{\sU(U)}{\mathcal{I}(U)\sU(U)} \]
by the right-module version of Lemma \ref{Revwotimes}. Because of the functorial nature of this construction, this defines a presheaf $\iota_+\sN$ on $\B$.

\begin{lem} Let $\sN$ be a coadmissible $\sW$-module on $Y$.
\be
\item For every $V \subset  U$ in $\B$, there is a natural isomorphism of $\sU(V)$-modules
\[(\iota_+\sN)(U) \underset{\sU(U)}{\w\otimes} \sU(V) \stackrel{\cong}{\longrightarrow} (\iota_+\sN)(V).\]
\item $\iota_+\sN$ is a sheaf on $\B$.
\ee
\end{lem}
\begin{proof} Since $\sN$ is coadmissible, by the right-module version of \cite[Theorem 9.4]{DCapOne} there is a natural isomorphism 
\[\sN(U\cap Y) \underset{\sW(U\cap Y)}{\w\otimes} \sW(V\cap Y) \stackrel{\cong}{\longrightarrow} \sN(V \cap Y)\]
of $\sW(V\cap Y)$-modules. The associativity of $\w\otimes$ given by the right-module version of \cite[Proposition 7.4]{DCapOne}  induces an isomorphism
\[ \sN(U \cap Y) \underset{\sW(U \cap Y)}{\w\otimes} \frac{ \sU(V) }{\I(V) \sU(V)} \stackrel{\cong}{\longrightarrow} \sN(V \cap Y) \underset{\sW(V \cap Y)}{\w\otimes} \frac{ \sU(V) }{\I(V) \sU(V)}.  \]
Since $\I$ is coherent and $\w\otimes$ is right exact, there is an isomorphism of left $\sW(U\cap Y)$-modules
\[ \frac{\sU(U)}{\I(U) \sU(U)} \underset{\sU(U)}{\w\otimes}\sU(V) \cong \frac{\sU(V)}{\I(V) \sU(V)}.\]
Substituting this expression for $\sU(V) / \I(V)\sU(V)$ into the previous isomorphism and applying \cite[Proposition 7.4]{DCapOne} again gives part (a). Now  $(\iota_+\sN)|_{U_w}$ is isomorphic to $\Loc( (\iota_+\sN)(U))$ for every $U \in \B$ by part (a), which is a sheaf on $U_w$ by Theorem \ref{RevLocaff}, and part (b) follows. \end{proof}

\begin{defn} Let $\sN$ be a coadmissible $\sW$-module on $Y$. We call the canonical extension of $\iota_+\sN$ to a sheaf $\iota_+\sN$ on $X_{\rig}$ given by \cite[Theorem 9.1]{DCapOne} the \emph{push-forward of $\sN$ along $\iota$}.
\end{defn}

\begin{prop} $\iota_+$ is a functor from coadmissible $\sW$-modules on $Y$ to coadmissible $\sU$-modules on $X$.
\end{prop}
\begin{proof} The functorial nature of $\iota_+$ is clear, and $\iota_+\sN$ is a coadmissible $\sU$-module by \cite[Theorem 9.4]{DCapOne} and part (a) of the Lemma.
\end{proof}

\subsection{The pullback functor \ts{\iota^\natural}}\label{SheafyInat}

Let $U \in \B$ and let $L = \sL(U)$. Recall from Subsection \ref{TransBi} the pullback $L_Y$ of $L$ to $Y$. Then $\sL_Y(U\cap Y) = L_Y$ by Lemma \ref{LXsmooth}, so 
\[ \sU(U) = \w{U(L)} \qmb{and} (\iota_\ast \sW)(U) = \sW(U \cap Y) = \w{U(L_Y)}.\]
For any sheaf of $\sU$-modules $\sM$ on $X_{\rig}$, $\sM(U)$ is a $\sU(U)$-module and 
\[ (\iota^\natural_\ast \sM)(U) := \sM(U)[\I(U)]\]
is a $(\iota_\ast \sW)(U)$-module as we saw in Subsection \ref{FrechKash}. This construction defines a presheaf $\iota^\natural_\ast \sM$ of $\iota_\ast \sW$-modules on $\B$.

\begin{lem} $\iota^\natural_\ast \sM$ is a sheaf of $\iota_\ast \sW$-modules on $\B$.
\end{lem}
\begin{proof} Let $\{U_i\}$ be an admissible affinoid covering of $U \in \B$, and let $\{W_{ijk}\}_k$ be an admissible affinoid covering of $U_i \cap U_j$ for each $i,j$. Since $\sM$ is a sheaf on $X_{\rig}$, the sequence
\[0 \to \sM(U) \to \prod \sM(U_i) \rightrightarrows \prod \sM(W_{ijk})\]
is exact, by \cite[Definition 9.1]{DCapOne}. Therefore
\[0 \to \sM(U)[I] \to \prod \sM(U_i)[I] \rightrightarrows \prod \sM(W_{ijk})[I]\]
is also exact, where $I := \I(U)$. But $I$ generates $\I(V)$ as an $\O(V)$-module for all $V \in U_w$ because $\I$ is coherent. Therefore
\[0 \to \sM(U)[\I(U)] \to \prod \sM(U_i)[\I(U_i)] \rightrightarrows \prod \sM(W_{ijk})[\I(W_{ijk})]\]
is exact and hence $\iota^\natural_\ast \sM$ is a sheaf on $\B$. 
\end{proof}

By \cite[Theorem 9.1]{DCapOne}, $\iota^{\natural}_\ast \sM$ extends to a sheaf $\iota^{\natural}_\ast \sM$ of $\iota_\ast \sW$-modules on $X_{\rig}$. 

\begin{defn} Let $\sM$ be a $\sU$-module on $X$. The \emph{pull-back of $\sM$ along $\iota$} is 
\[ \iota^{\natural} \sM := \iota^{-1} (\iota^{\natural}_\ast \sM).\]
This is a sheaf of $\sW$-modules on $Y$.
\end{defn}

Since $\iota_\ast \sW$ is supported on $Y$, $\iota^\natural_\ast \sM$ is also supported on $Y$, so Theorem \ref{ShSuppClsdSub} implies that there is a natural isomorphism
\[ \iota_\ast (\iota^\natural \sM) \stackrel{\cong}{\longrightarrow} \iota^\natural_\ast \sM.\]
In particular, we see that $(\iota^\natural \sM)(U \cap Y) \cong (\iota^\natural_\ast \sM)(U) = \sM(U)[\I(U)]$ for every open affinoid subvariety $U$ of $X$. 

\subsection{Locally topologically nilpotent actions}\label{LtnX1/g} In this Subsection, we explain what it means for a local section of $\O$ to act locally topologically nilpotently on local sections of coadmissible $\w{\sU(\sL)}$-modules in geometric terms. 

We suppose that $X$ is affinoid, $\A$ is an affine formal model in $\O(X)$, $\L$ is a smooth $(\R,\A)$-Lie algebra and that $g \in \O(X)$ is a non-zero element such that $\L \cdot g \subset \A$. Recall the sheaf $\cS := \hsULK$ on the $\L$-admissible $G$-topology $X_w(\L)$ on $X$ from \cite[\S 3.2]{DCapOne}. The rational subdomain $X(1/g)$ is $\L$-accessible by \cite[Definition 4.6]{DCapOne} and is therefore $\L$-admissible by \cite[Proposition 4.6]{DCapOne}.

\begin{prop} Let $M$ be a finitely generated $\cS(X)$-module, and let $M_\infty$ be a generating set for $M$ as an $\cS(X)$-module. Then the following are equivalent:
\be 
\item $M \otimes_{\cS(X)}\cS(X(1/g))=0$,
\item $M\langle t\rangle (1 - gt)  = M\langle t\rangle$,
\item $\lim\limits_{k\to \infty} v g^k = 0$ for all $v \in M_\infty$.
\ee\end{prop}
\begin{proof} Let $\L_2 := \A\langle t \rangle \otimes_{\A} \L$. In \cite[Proposition 4.2]{DCapOne} we constructed an $(R, \A\langle t\rangle)$-Lie algebra structure $(\L_2, \sigma_2)$. By \cite[Proposition 4.3(c)]{DCapOne}, there is an isomorphism of left $\hK{U(\L_2)}$-modules
\[ \cS(X(1/g)) \cong \frac{\hK{U(\L_2)}}{\hK{U(\L_2)} (1 - gt)}.\]

(a) $\Leftrightarrow$ (b). Letting $N := M \otimes_{\cS(X)} \hK{U(\L_2)}$, we see that $M \otimes_{\cS(X)} \cS(X(1/g)) \cong N / N(1-gt)$ by the above. Now $N \cong M \langle t\rangle$ as an $A\langle t \rangle$-module by \cite[Lemma 4.4]{DCapOne}.

(b) $\Rightarrow$ (c). Fix an element $v\in M_\infty$. Since $M\langle t\rangle(1 - gt)  = M \langle t \rangle$, there is an element $\sum_{i=0}^\infty t^im_i\in M\langle t\rangle$ such that 
\[\sum_{i=0}^\infty t^im_i \cdot (1 - gt)=v.\] 
Then, comparing the coefficients of $t^i$, we see that $m_0=v$ and $m_i = m_{i-1}g$ for each $i \geq 1$. Thus $m_k=vg^k$ for each $k\ge 0$. Since $\sum_{i=0}^\infty t^im_i\in M\langle t\rangle$ we must have $vg^k\to 0$ as $k\to \infty$. 

(c) $\Rightarrow$ (b) By assumption, $\sum_{i=0}^\infty t^i g^iv$ lies in $M\langle t\rangle$ and
\[\sum_{i=0}^\infty t^i g^iv \cdot ( 1 - gt ) = v,\]
so $v\in M \langle t \rangle (1 - gt) $ for any $v \in M_\infty$. It follows from \cite[Proposition 4.3(c)]{DCapOne} that the element $1 - gt$ is \emph{normal} in $\h{U(\L_2)_K}$. Hence $M\langle t\rangle (1 - gt)$ is an $\h{U(\L_2)_K}$-submodule of $M\langle t\rangle$ which contains a generating set for $M$. Since $M$ generates $M\langle t \rangle$ as  an $\h{U(\L_2)_K}$-module, we see that $M\langle t \rangle (1 - gt)  = M\langle t \rangle$.
\end{proof}

\begin{cor} Suppose that $X$ is affinoid, $g\in \O(X)$ is non-zero and $\sL(X)$ has a smooth $(R,\A)$-Lie lattice. The following are equivalent for a coadmissible $\w{\sU(\sL)}$-module $\sM$ on $X$:
\be\item $\sM(X(1/g)) = 0$,
\item $g$ acts locally topologically nilpotently on $\sM(X)$.
\ee\end{cor}
\begin{proof} By \cite[Lemma 7.6(a)]{DCapOne} there is a smooth $(R,\A)$-Lie lattice $\L$ in $\sL(X)$ such that $\L\cdot g\subset A$. Let $X' = X(1/g)$, and write $U := \w{\sU(\sL)}(X)$ and $U' := \w{\sU(\sL)}(X')$. Let $M = \sM(X)$ and $M' = \sM(X')$ so that $M' \cong M \w\otimes_U U'$ by \cite[Theorem 9.4]{DCapOne}. Let $\cS_n$ be the sheaf $\hK{U(\pi^n \L)}$ on $X_{ac}(\pi^n \L)$ and note that $X'$ lies in $X_{ac}(\pi^n \L)$ for all $n\geq 0$. Write $U_n = \cS_n(X)$ and $U'_n = \cS_n(X')$; then 
\[ U \cong \invlim U_n \qmb{and} U' \cong \invlim U'_n\]
give presentations of the Fr\'echet-Stein algebras $U$ and $U'$, respectively. In the language of \cite[\S 3]{ST}, $(M \otimes_U U'_n)_n$ is a coherent sheaf for $U'$ whose module of global sections $\invlim M \otimes_U U'_n$ is isomorphic to $M'$ by the definition of $M \w\otimes_U U'$. Now
\[ M' \otimes_{U'}U'_n \cong M \otimes_U U'_n \qmb{for all} n\geq 0\]
by \cite[Corollary 3.1]{ST}, so $M' = 0$ if and only if $M \otimes_U U'_n = 0$ for all $n\geq 0$.

Now $M \otimes_U U_n$ is a finitely generated $U_n$-module since $M$ is a coadmissible $U$-module, and the image of $M$ in $M\otimes_U U_n$ is a $U_n$-generating set. Hence for any $n\geq 0$, $(M \otimes_U U_n) \otimes_{U_n} U_n' = 0$ if and only if $g$ acts locally topologically nilpotently on the image of $M$ in $M \otimes_U U_n$ by the Proposition. Since the topology on $M \cong \invlim M \otimes_UU_n$ is the inverse limit topology, this is in turn equivalent to $g$ acting locally topologically nilpotently on $M$.
\end{proof}

\subsection{Support and \ts{M_\infty(I)}}\label{SuppMinfty}

Armed with Corollary \ref{LtnX1/g}, we can now explain the geometric meaning of the submodule $M_\infty(I)$ that featured in Theorem \ref{KashAffinoids}.

\begin{thm} Suppose that $X$ is affinoid and that $\sL(X)$ has an $\I(X)$-standard basis.  The following are equivalent for a coadmissible $\w{\sU(\sL)}$-module $\sM$ on $X$:
\be\item $\sM$ is supported on $Y$,
\item $\sM(X) = \sM(X)_\infty(\I(X))$.
\ee\end{thm}
\begin{proof} Write $A = \O(X)$, $M = \sM(X)$ and choose a generating set $f_1,\ldots, f_r$ for $I = \I(X)$. The complement of $Y$ in $X$ is an admissible open subset which admits a Zariski covering 
\[X \backslash Y = D(f_1) \cup \cdots \cup D(f_r).\]
It follows from \cite[Corollary 9.1.4/7]{BGR} that this covering is admissible, so $\sM$ is supported on $Y$ if and only if it is supported on the closed analytic subset $V(f_i)$ for all $i=1,\ldots, r$. Now
\[ D(f_i) = \bigcup_{n\geq 0} X(\pi^n/f_i)\]
is an admissible covering of the Zariski open subset $D(f_i)$ of $X$, so $\sM$ is supported on $Y$ if and only if its restriction to $X(\pi^n /f_i)$ is zero for all $n\geq 0$ and all $i=1,\ldots, r$. Since $\sM$ is coadmissible, by \cite[Theorem 9.4]{DCapOne} this is equivalent to $\sM(X(\pi^n/f_i)) = 0$ for all $n\geq 0$ and all $i$, which is in turn equivalent to $f_i/\pi^n$ acting locally topologically nilpotently on $M$ for all $n\geq 0$ and all $i$ by Corollary \ref{LtnX1/g}. This is the same as $M = M_\infty(Kf_i)$ for all $i=1,\ldots, r$. But 
\[M_\infty(I) = M_\infty(Af_1) \cap \cdots \cap M_\infty(Af_r)\]
by Lemma \ref{Minfty}(b), and $M_\infty(Kf_i) = M_\infty(Af_i)$ by Corollary \ref{Minfty}. So $M = M_\infty(Kf_i)$ for all $i$ if and only if $M = M_\infty(I)$. 
\end{proof}

\subsection{The algebra \ts{\sV}}\label{SheafyV}
Suppose that $X$ is affinoid and $\sL(X)$ has an $\I(X)$-standard basis. Choose a generating set $F = \{f_1,\ldots, f_r\}$ for $\I(X)$ satisfying Definition \ref{Basis}. Recall that the centraliser $C_{\sL(X)}(F)$ of $F$ in $\sL(X)$ is a free $(K,\O(X))$-Lie algebra by Proposition \ref{Basis}.

\begin{prop} Let $C = C_{\sL(X)}(F)$ and $\sV := \w{\sU(C)}$. Then there is an isomorphism
\[ \sV / F \sV \cong \iota_\ast \sW\]
of sheaves of $K$-algebras on $X$.
\end{prop}
\begin{proof} Let $U$ be an affinoid subdomain of $X$ and let $V = U \cap Y$. By Theorem \ref{RevLiealg}, there is an isomorphism $\sW(V) \cong \w{U(\sL_Y(V))}$, and 
\[ \sL_Y(V) \cong \frac{N_{\sL(U)}\left(\I(U)\right)}{\I(U)\sL(U)}\]
by Lemma \ref{LXsmooth}. Since $\sL$ is a coherent $\O$-module, Proposition \ref{Basis} implies that there is a natural isomorphism
\[\O(U) \otimes_{\O(X)} C_{\sL(X)}(F) \stackrel{\cong}{\longrightarrow} C_{\sL(U)}(F),\]
so there is a natural isomorphism $\sV(U) \cong \w{U(C_{\sL(U)}(F))}$. Now Lemma \ref{VoverFV} implies that the sequence of $\sV(U)$-modules
\[ \sV(U)^r \to \sV(U) \to \sW(U \cap Y) \to 0\]
where the first arrow sends $(v_1,\ldots, v_r)$ to $\sum_{i=1}^r f_i v_i$, is exact. The result now follows from \cite[Theorem 9.1]{DCapOne}.
\end{proof}

\subsection{Theorem}\label{InatCoadm} Suppose that $X$ is affinoid and that $\sL(X)$ has an $\I(X)$-standard basis. Let $\sM$ be a coadmissible $\sU$-module supported on $Y$ and let $X'$ be an affinoid subdomain of $X$. Then the natural map
\[ \sM(X)[\I(X)] \underset{\sW(Y)}{\w\otimes} \sW(Y') \longrightarrow \sM(X')[\I(X')]\]
is an isomorphism, where $Y' = Y \cap X'$.
\begin{proof} By Proposition \ref{SheafyV}, it is enough to show that
\[ \alpha : \sM(X)[\I(X)] \underset{\sV(X)}{\w\otimes} \sV(X') \longrightarrow \sM(X')[\I(X')]\]
is an isomorphism. Write $M = \sM(X)$, $M' = \sM(X')$, $U = \sU(X)$, $U' = \sU(X')$, $V = \sV(X)$, $V' = \sV(X')$, $I = \I(X)$ and $I' = \I(X')$. Note that with this notation $M' \cong M \w\otimes_U U'$ by \cite[Theorem 9.4]{DCapOne}, because $\sM$ is coadmissible. Now there is a natural commutative diagram 
\[\xymatrix{ (M[I] \underset{V}{\w\otimes} V') \underset{V'}{\w\otimes} U'\ar[rrrr]^{\alpha \w\otimes 1}\ar[d]_{\cong} & & && M'[I']\underset{V'}{\w\otimes} U'\ar[d]^{\epsilon_{M'}} \\
(M[I] \underset{V}{\w\otimes} U) \underset{U}{\w\otimes} U' \ar[d]_{\epsilon_M \w\otimes 1}& &  &  & M'_{\infty}(I') \ar[d]\\
M_\infty(I)\underset{U}{\w\otimes}U' \ar[rr] & & M \underset{U}{\w\otimes}U'  \ar[rr]_{\cong} & & M' }\]
Since $\sM$ is supported on $Y$ by assumption, Theorem \ref{SuppMinfty} implies that 
\[ M = M_\infty(I) \qmb{and} M' = M'_\infty(I').\]
Hence the unmarked arrows in the above diagram are equalities. Also, $\epsilon_M$ and $\epsilon_{M'}$ are isomorphisms by Theorem \ref{KashAffinoids}, so the diagram shows that $\alpha \w\otimes 1$ is an isomorphism. But $U'$ is a faithfully c-flat $V'$-module by Corollary \ref{Lieflat} and the right-module version of \cite[Proposition 7.5(c)]{DCapOne}, so $\alpha$ is an isomorphism as required.
\end{proof}

We conjecture that this result also holds when the condition that $\sM$ is supported on $Y$ is removed, but are unable to prove this at present. 

\subsection{Kashiwara's Theorem for right \ts{\w{\sU(\sL)}} modules}\label{KashiwaraGeneral}
We can now state and prove our version of Kashiwara's equivalence for right modules. 

\begin{thm} Let $X$ be a rigid analytic variety and let $\sL_X$ be a Lie algebroid on $X$. Let $\iota : Y \hookrightarrow X$ be the inclusion of a closed, analytic subvariety such that $\theta \circ \iota^\ast \rho : \iota^\ast \sL \to \N_{Y/X}$ is surjective, and such that $\N_{Y/X}$ is locally free.
\be
\item If $\sM$ is a coadmissible right $\w{\sU(\sL_X)}$-module supported on $Y$, then $\iota^\natural \sM$ is a coadmissible right $\w{\sU(\sL_Y)}$-module.
\item The restriction of $\iota^\natural$ to coadmissible right $\w{\sU(\sL_X)}$-modules supported on $Y$ is right adjoint to $\iota_+$.
\item These functors induce an equivalence of abelian categories
\[\left\{ 
				\begin{array}{c} 
					co\hspace{-0.1cm}-\hspace{-0.1cm}admissible\hspace{0.1cm} right \hspace{0.1cm}\\
					\w{\sU(\sL_Y)}-\hspace{-0.1cm}modules\hspace{0.1cm}
				\end{array}
\right\} \cong \left\{
				\begin{array}{c}
				 co\hspace{-0.1cm}-\hspace{-0.1cm}admissible \hspace{0.1cm} right \\ 
				 \w{\sU(\sL_X)} \hspace{-0.1cm}-\hspace{-0.1cm}modules\hspace{0.1cm} \hspace{0.1cm} supported \hspace{0.1cm} on \hspace{0.1cm} Y
				\end{array}
\right\}.\]
\ee\end{thm}
\begin{proof}
Let $\I$ be the radical coherent ideal of $\O_X$ consisting of functions vanishing on $Y$, and let $\B$ be the set of open affinoid subvarieties $U$ of $X$ such that either $\I(U) = \I(U)^2$ or $\sL(U)$ has an $\I(U)$-standard basis. By Theorem \ref{StdBasis}, $\B$ is a basis for the strong $G$-topology on $X$.

(a) Pick an admissible covering $\{U_j\}$ of $X$ with each $U_j \in \B$. By passing to a finite refinement of each $U_j$, we may assume that each $U_j$ is connected. Then $\{ U_j \cap Y\}$ is an admissible affinoid covering of $Y$, so it is enough to show that $(\iota^\natural \sM)|_{U_j \cap Y}$ is a coadmissible $\sW|_{U_j \cap Y}$-module for each $j$. If $\iota_j$ denotes the inclusion $U_j \cap Y \hookrightarrow U_j$, then it follows from Theorem \ref{ShSuppClsdSub} and \cite[Theorem 9.1]{DCapOne} that 
\[ (\iota^\natural \sM)|_{U_j \cap Y} \cong \iota_j^\natural(\sM|_{U_j}).\]
We may thus assume that $X \in \B$ and that $X$ is connected; Lemma \ref{BasisB} then implies that either $Y = \emptyset$, or $Y = X$ or $\sL(X)$ has an $\I(X)$-standard basis. When $Y$ is empty, $\I(X) = \O(X)$ and the definition of $\iota^\natural \sM$ shows that $\iota^\natural \sM = 0$, which is coadmissible. When $Y = X$, the Lie algebroid $\sL_Y$ is equal to $\sL_X$, $\sW = \sU$ and $\iota^\natural \sM = \sM$ is coadmissible. Suppose that $\sL(X)$ has an $\I(X)$-standard basis. Now $(\iota^\natural \sM)(X) = \sM(X)[\I(X)]$ is a coadmissible $\sW(X)$-module by Theorem \ref{KashAffinoids}(a), and Theorem \ref{InatCoadm} implies that the natural map
\[ \iota_\ast \Loc\left(\sM(X)[\I(X)]\right) \longrightarrow \iota^\natural_\ast \sM\]
of $\iota_\ast\sW$-modules is an isomorphism. Now Theorem \ref{ShSuppClsdSub} implies that 
\[ \Loc \left(\sM(X)[\I(X)]\right) \longrightarrow \iota^\natural \sM\]
is an isomorphism, so $\iota^\natural\sM$ is a coadmissible $\sW$-module by \cite[Definition 8.3]{DCapOne}. 

(b) Let $\sN$ be a coadmissible $\sW$-module on $Y$ and let $\sM$ be a coadmissible $\sU$-module on $X$ which is supported on $Y$.  By \cite[Theorem 9.1]{DCapOne} there is a natural isomorphism
\[ \Hom_{\sU}(\iota_+ \sN, \sM) \stackrel{\cong}{\longrightarrow} \Hom_{\sU|_{\B}}( (\iota_+\sN)|_{\B}, \sM|_{\B}), \]
and by \cite[Theorem 9.1]{DCapOne} and Theorem \ref{ShSuppClsdSub} there is a natural isomorphism
\[ \Hom_{\sW}(\sN, \iota^\natural \sM) \stackrel{\cong}{\longrightarrow} \Hom_{\iota_\ast \sW|_{\B}} ( (\iota_\ast \sN)|_{\B}, (\iota^\natural_\ast\sM)|_{\B}).\]
For any morphism $\alpha :  (\iota_+\sN)|_{\B} \to \sM|_{\B}$ of $\sU|_{\B}$-modules and any $U \in \B$, let
\[ \beta(U) : \sN(U \cap Y) \to \sM(U)[\I(U)]\]
be given by $\beta(U)(x) = \alpha(U)(x \w\otimes 1)$. This is a map of $\sW(U \cap Y)$-modules. Since $\alpha$ is a morphism of $\sU|_{\B}$-modules, the diagram
\[\xymatrix{   \sN(U \cap Y) \underset{\sW(U \cap Y)}{\w\otimes} \frac{\sU(U)}{\I(U)\sU(U)} \ar[r]^(0.7){\alpha(U)}\ar[d] & \sM(U) \ar[d] \\
\sN(V \cap Y) \underset{\sW(V \cap Y)}{\w\otimes} \frac{\sU(V)}{\I(V)\sU(V)} \ar[r]_(0.7){\alpha(V)} & \sM(V) }\]
commutes, so the diagram
\[\xymatrix{  \sN(U \cap Y) \ar[r]^{\beta(U)} \ar[d] & \sM(U)[\I(U)] \ar[d] \\
\sN(V \cap Y) \ar[r]_{\beta(V)} & \sM(V) [\I(V)] }\]
also commutes because
\[ \beta(U)(x)|_{V} = \alpha(U)(x \w\otimes 1)|_{V} = \alpha(V)( (x\w\otimes 1)|_{V}) = \alpha(V)(x|_{V} \w\otimes 1) = \beta(V) (x|_{V}).\]
Thus $\beta : (\iota_\ast \sN)|_{\B} \to (\iota^\natural_\ast\sM)|_{\B}$ is a morphism of $(\iota_\ast \sW)|_{\B}$-modules, and applying Theorem \ref{KashAffinoids}(b) we obtain a bi-functorial injection
\[ \Phi(\sN,\sM) : \Hom_{\sU|_{\B}}( (\iota_+\sN)|_{\B}, \sM|_{\B}) \hookrightarrow \Hom_{\iota_\ast \sW|_{\B}} ( (\iota_\ast \sN)|_{\B}, (\iota^\natural_\ast\sM)|_{\B})\]
by setting $\Phi(\sN,\sM)(\alpha) := \beta$. Since a $\sU(U)$-linear morphism 
\[\sN(U \cap Y) \underset{\sW(U \cap Y)}{\w\otimes} \frac{\sU(U)}{\I(U)\sU(U)} \to \sM(U)\] is determined by its restriction to the image of $\sN(U\cap Y)$, $\Phi(\sN,\sM)$ is actually a bijection. Putting everything together gives an adjunction
\[\Hom_{\sU}(\iota_+ \sN, \sM) \stackrel{\cong}{\longrightarrow}  \Hom_{\sW}(\sN, \iota^\natural \sM).\]

(c) Note that the definition of $\iota_+\sN$ for a coadmissible $\sW$-module $\sN$ on $Y$ shows that $\iota_+\sN$ is supported on $Y$. By part (b) we have an adjunction $(\iota_+, \iota^\natural)$ between the categories of interest. Let $\sM$ be a coadmissible $\sU$-module supported on $Y$ and consider the local sections of the co-unit morphism 
\[ \epsilon_{\sM}(U) : (\iota_+ \iota^\natural \sM)(U) \to \sM(U)\]
for some connected $U \in \B$. If $U \cap Y$ is empty then $\sM(U) = 0$ since $\sM$ is supported on $Y$. If $U \cap Y = U$ then $(\iota_+\iota^\natural \sM)(U) = \sM(U)$ and $\epsilon_{\sM}(U)$ is the identity map. By Lemma \ref{BasisB} we can therefore assume that $\sL(U)$ has an $\I(U)$-standard basis. Then $\I(U)$ acts locally topologically nilpotently on $\sM(U)$ by Theorem \ref{SuppMinfty} because $\sM|_{U}$ is supported on $U \cap Y$. Hence $\epsilon_{\sM}(U)$ is an isomorphism by Theorem \ref{KashAffinoids}(c) for all $U \in \B$ and hence $\epsilon_{\sM}$ is an isomorphism. A similar argument shows that $\iota_+$ reflects isomorphisms. So $\iota_+$ and $\iota^\natural$ are mutually inverse equivalences of categories by Proposition \ref{AffKash}.\end{proof}

\section{Main Results}\label{Main}

\subsection{Kashiwara's Theorem for left \ts{\w{\sU(\sL)}} modules}\label{KashGenLeft}
In the setting of Theorem \ref{KashiwaraGeneral}, recall from Section \ref{LRSwitch} the invertible sheaves
\[\Omega_{\sL_X}:=\mathpzc{Hom}_{\O_X}(\bigwedge^{\rk \sL_X} \sL_X, \O_X) \quad \qmb{and} \quad\Omega_{\sL_Y}:=\mathpzc{Hom}_{\O_Y}(\bigwedge^{\rk \sL_Y} \sL_Y, \O_Y)\]
that implement the side-switching operations on $X$ and $Y$ respectively. We obtain Kashiwara's equivalence for left  $\w{\sU(\sL_X)}$-modules by combining the equivalence for right $\w{\sU(\sL_X)}$-modules given by Theorem \ref{KashiwaraGeneral} together with appropriate side-switching operations. 
\begin{thm} Let $X$ be a rigid analytic variety and let $\sL_X$ be a Lie algebroid on $X$. Let $\iota : Y \hookrightarrow X$ be the inclusion of a closed, analytic subset such that $\theta \circ \iota^\ast \rho : \iota^\ast \sL \to \N_{Y/X}$ is surjective, and such that $\N_{Y/X}$ is locally free. Then the functors $\iota_+$ and $\iota^\natural$ given by
\[\begin{array}{rcl} \iota_+ \sN &:=& \mathpzc{Hom}_{\O_X}\left(\Omega_{\sL_X} , \iota_+(\Omega_{\sL_Y} \otimes_{\O_Y} \sN)\right)\qmb{and}\\

\iota^\natural \sM &:=& \mathpzc{Hom}_{\O_Y}\left(\Omega_{\sL_Y} , \hspace{0.05cm}\iota^\natural\hspace{0.05cm}(\Omega_{\sL_X} \otimes_{\O_X} \sM)\right)\end{array}\]
are mutually inverse equivalences of abelian categories
\[\left\{ 
				\begin{array}{c} 
					co\hspace{-0.1cm}-\hspace{-0.1cm}admissible\hspace{0.1cm} left \hspace{0.1cm}\\
					\w{\sU(\sL_Y)}-\hspace{-0.1cm}modules\hspace{0.1cm}
				\end{array}
\right\} \cong \left\{
				\begin{array}{c}
				 co\hspace{-0.1cm}-\hspace{-0.1cm}admissible \hspace{0.1cm} left \\ 
				 \w{\sU(\sL_X)} \hspace{-0.1cm}-\hspace{-0.1cm}modules\hspace{0.1cm} \hspace{0.1cm} supported \hspace{0.1cm} on \hspace{0.1cm} Y
				\end{array}
\right\}.\]
\end{thm}
\begin{proof}This follows immediately from Theorems \ref{wULswitch} and \ref{KashiwaraGeneral}, once it is observed that the side-switching functors preserve the support condition: $\Omega_{\sL_X} \otimes_{\O_X} \sM$ and $\mathpzc{Hom}_{\O_X}(\Omega_{\sL_X}, \sM)$ are supported on $Y$ whenever $\sM$ is supported on $Y$.
\end{proof}
Just like in the classical case --- see, for example, \cite[Definition 1.3.3]{HTT} --- it is possible to realise the push-forward functor $\iota_+$ for left $\w{\sU(\sL_X)}$-modules as tensoring on the left with an appropriate bimodule $\w{\sU(\sL)}_{X \leftarrow Y}$, but we do not give these details here. 

Theorem \ref{Kashiwara} from the Introduction follows immediately from Theorem \ref{KashGenLeft} because if $X$ and $Y$ are both smooth then the second fundamental sequence is exact by \cite[Proposition 2.5]{BLR3}, so the normal bundle $\N_{Y/X}$ is locally free and $\theta : \iota^\ast \T \to \N_{Y/X}$ is surjective. More generally, for $X$ and $Y$ still both smooth, the conditions of the Theorem hold whenever the anchor map $\sL_X\to\T$ is surjective, such as when $\sL_X$ is an Atiyah algebra in the sense of \cite{BeilinsonSchechtman}. There are also interesting cases where the Theorem applies but the anchor map is not surjective as illustrated in the following example. 

\begin{example}[cf \cite{MOLogQ}] Suppose that $X$ is a smooth rigid analytic variety with simple normal crossings divisor $D$. There is a stratification of $X$ defined as follows: $X_0=X\backslash D$; $X_1$ is the non-singular part of $D$; and $X_{k+1}$ the non-singular part of $\overline{X_k}\backslash X_k$ for each $k\geq 1$. 

We say a smooth closed analytic subvariety $Y$ of $X$ meets $D$ \emph{transversely} if $Y$ meets each stratum $X_k$ transversely; i.e. if $T_pY+T_pX_k=T_pX$ for each $p\in Y\cap X_k$. 

The conditions of Theorem \ref{KashGenLeft} hold for the logarithmic tangent sheaf $\sL_X=\T_X(-\log D)$ whenever $Y$ meets $D$ transversely.
\end{example}

\subsection{A special class of \ts{\w{U(L)}}-modules}\label{RestToUL} From now on, all modules will be \emph{left} modules, unless explicitly stated otherwise.

\begin{lem} Let $\A$ be an admissible $\R$-algebra and let $\L$ be a coherent $(\R,\A)$-Lie algebra. Suppose that the $U(\L)$-module $\M$ is finitely generated as an $\A$-module. Then the natural map $\M \to \h{U(\L)} \otimes_{U(\L)} \M$ is an isomorphism.
\end{lem}
\begin{proof} The algebra $\A$ is $\pi$-adically complete, so the finitely generated $\A$-module $\M$ is also $\pi$-adically complete by \cite[\S 3.2.3(v)]{Berth}: the canonical map $\M \to \h{\M}$ is an isomorphism. But this map factorises as $\M \to \h{U(\L)} \otimes_{U(\L)} \M \to \h{\M}$ and the second map is an isomorphism by \cite[\S 3.2.3(iii)]{Berth} because $\M$ is a finitely generated module over the Noetherian ring $U(\L)$. The result follows.
\end{proof}

\begin{prop} Let $A$ be a $K$-affinoid algebra and let $L$ be a $(K,A)$-Lie algebra which admits a smooth Lie lattice. Suppose that the $U(L)$-module $M$ is finitely generated as an $A$-module. Then 
\be
\item $M$ is a coadmissible $\w{U(L)}$-module, and
\item the natural map $M \to \w{U(L)} \otimes_{U(L)} M$ is a $\w{U(L)}$-linear isomorphism.
\ee\end{prop}
\begin{proof} (a) Let $\L$ be a smooth $\A$-Lie lattice in $L$ for some affine formal model $\A$ in $A$. Let $S$ be a finite generating set for $M$ as an $A$-module and let $X$ be a finite generating set for $\L$ as an $\A$-module. Then $\M := \A S$ generates $M$ as a $K$-vector space and $XS$ is finite, so $\pi^t XS \subset \M$ for some integer $t$. Fix $n\geq t$; then  $(\pi^n \L) \M \subset \M$ so $\M$ is a $U(\pi^n \L)$-module which is finitely generated as an $\A$-module.  Hence $\M \to \h{U(\pi^n \L)} \otimes_{U(\pi^n \L)} \M$ is an isomorphism by the Lemma. But $U(L) \cong K \otimes_\R U(\pi^n \L)$ and $M \cong K \otimes_\R \M$ so $M \to \hK{U(\pi^n \L)} \otimes_{U(L)} M$ is an isomorphism whenever $n \geq t$. In particular, by transport of structure $M$ is naturally a $U_n:= \hK{U(\pi^n \L)}$-module whenever $n \geq t$.

Now, consider the commutative diagram of $U := U(L)$-modules
\[\xymatrix{ 
M \ar[r]\ar[d] & U_n \otimes_UM \ar[d] \\
U_n \otimes_{U_{n+1}} M \ar[r] & U_n \otimes_{U_{n+1}} (U_{n+1} \otimes_U M). }\]
The horizontal arrows are isomorphisms whenever $n \geq t$ by the above, and the vertical arrow on the right is an isomorphism by the associativity of tensor product. Hence $U_n\otimes_{U_{n+1}} M \to M$ is a bijection with inverse $m \to 1 \otimes m$ whenever $n \geq t$. So $(M)$ is a coherent sheaf for $U_\bullet$ and $M = \invlim M$ is a coadmissible $\w{U}$-module.

(b) If $N$ is a finitely generated $U$-module then $(U_n \otimes_U N)$ is a coherent sheaf for $U_\bullet$, and it follows from \cite[Theorem B]{ST} that the functor $N \mapsto \invlim U_n \otimes_U N$ is right exact on finitely generated $U$-modules. There is a natural map $\eta_N : \w{U} \otimes_U N \to \invlim U_n \otimes_U N$ which is an isomorphism when $N = U$. So $\eta_N$ is always an isomorphism by the Five Lemma. Now consider the natural commutative triangle of $U$-modules:
\[\xymatrix{ 
M \ar[r] \ar[dr]_\alpha & \invlim U_n \otimes_{U} M \\
& \w{U} \otimes_{U} M.\ar[u]_{\eta_M}
}\]
The horizontal map is an isomorphism by the discussion in the proof of  part (a), and $M$ is a finitely generated $U$-module, so $\eta_M$ is an isomorphism. Hence $\alpha$ is an isomorphism. Let $\beta : \w{U} \otimes_{U}M \to M$ be the action map; then $\beta \circ \alpha = 1_M$ so $\beta$ is also an isomorphism. So $\alpha$ is $\w{U}$-linear because its inverse $\beta$ is $\w{U}$-linear.\end{proof}

\subsection{\ts{\O}-coherent coadmissible \ts{\w{\sU(\sL)}}-modules}\label{OxCoh} It turns out that all $\O$-coherent coadmissible $\w{\sU(\sL)}$-modules are algebraic in the following precise sense.

\begin{thm} Let $\sL$ be a Lie algebroid on the rigid analytic space $X$. Then the forgetful functor 
\[\left\{ 
				\begin{array}{c} 
					\O-coherent \hspace{0.1cm} \\
					co\hspace{-0.1cm}-\hspace{-0.1cm}admissible\hspace{0.1cm} \w{\sU(\sL)}-\hspace{-0.1cm}modules\hspace{0.1cm}
				\end{array}
\right\} \longrightarrow \left\{
				\begin{array}{c}
				 \O-coherent \hspace{0.1cm} \\
				 U(\sL) \hspace{-0.1cm}-\hspace{-0.1cm}modules\hspace{0.1cm} 
				\end{array}
\right\}\]
is an equivalence of categories.\end{thm}
\begin{proof} The forgetful functor is faithful, so it will be enough to show that is (a) essentially surjective on objects and (b) full. 

(a) Suppose that $\sM$ is an $\O$-coherent $U(\sL)$-module and that $Z \subset Y$ are affinoid subdomains of $X$ such that $\sL(Y)$ admits a smooth Lie lattice.

Let $\sU := \w{\sU(\sL)}$, $U := U(\sL(Y))$, $U' := U(\sL(Z))$, $A := \O(Y)$, $A' := \O(Z)$, $M := \sM(Y)$ and $M' := \sM(Z)$, so that $\sU(Y) = \w{U}$ and $\sU(Z) = \w{U'}$. Then $M$ is a $U$-module which is finitely generated as an $A$-module, and $M' \cong A' \otimes_A M$ because $\sM$ is $\O$-coherent.  Because $U' \cong A' \otimes_A U$, the functors $U' \otimes_U -$ and $A' \otimes_A -$ are isomorphic on finitely presented $U$-modules by the Five Lemma. Hence the natural map $U' \otimes_U M \to M'$ is an isomorphism. Now $M$ is a coadmissible $\w{U}$-module and $M'$ is a coadmissible $\w{U'}$-module by Proposition \ref{RestToUL}(a), and the maps 
\[ M \to \w{U} \otimes_{U} M \qmb{and} M' \to \w{U'} \otimes_{U'} M\]
are isomorphisms by Proposition \ref{RestToUL}(b). Hence there are isomorphisms
\[ \w{U'} \underset{\w{U}}{\w\otimes} M \cong \w{U'} \underset{\w{U}}{\w\otimes} \left(\w{U} \underset{U}{\otimes} M\right) \cong \w{U'} \underset{U}{\otimes}  M \cong  \w{U'} \underset{U'}{\otimes} \left(U' \underset{U}{\otimes} M\right) \cong M'\]
because $\w\otimes$ is right exact by \cite[Proposition 7.5(a)]{DCapOne} and because $M$ is a finitely presented $U$-module. Therefore $\sM$ is a coadmissible $\sU$-module as required.

(b) We have to show that every $U(\sL)$-linear morphism between two $\O$-coherent $U(\sL)$-modules is $\sU$-linear. This is a local problem, so we may assume that $X$ is affinoid and $L := \sL(X)$ admits a smooth Lie lattice. Now if $f : M \to N$ is a $U := U(L)$-linear map between two finitely generated $\O(X)$-modules, then there is a commutative diagram of $U$-modules
\[\xymatrix{M \ar[r]^{f}\ar[d] & N \ar[d] \\ 
\w{U}\otimes_UM \ar[r]_{1 \otimes f} & \w{U}\otimes_U N.
}\]
It follows that $f$ is $\w{U}$-linear because the vertical arrows are $\w{U}$-linear isomorphisms by Proposition \ref{RestToUL}(b) and the bottom arrow is $\w{U}$-linear.
\end{proof}

\begin{cor} With the notation of the Theorem, a coadmissible $\w{\sU(\sL)}$-submodule of an $\O$-coherent coadmissible $\w{\sU(\sL)}$-module is $\O$-coherent.
\end{cor}
\begin{proof} This is a local problem, so we may assume that $X$ is a $K$-affinoid variety and that $\sL(X)$ admits a smooth Lie lattice. Let $\sM$ be the coadmissible $\sU := \w{\sU(\sL)}$-submodule; then $M := \sM(X)$ is an $\O(X)$-submodule of a finitely generated $\O(X)$-module by assumption and is therefore itself finitely generated over $\O(X)$ because $\O(X)$ is Noetherian.

Consider the coherent $\O$-module $\widetilde{M}$ associated to $M$. Since $M$ is an $\sL(X)$-module and $\sL(U) \cong \O(U) \otimes_{\O(X)} \sL(X)$ for every affinoid subdomain $U$ of $X$, we see that $\widetilde{M}$ is naturally an $\O$-coherent $U(\sL)$-module. So $\widetilde{M} \cong \sN$ as $U(\sL)$-modules for some $\O$-coherent coadmissible $\sU$-module $\sN$ by the Theorem. In particular, $M = \widetilde{M}(X) \cong \sN(X)$ as $U(\sL(X))$-modules. Because the restriction functor is full, $M \cong \sN(X)$ also as $\sU(X)$-modules by the Theorem. Therefore
\[\sM \cong \Loc(M) \cong \Loc(\sN(X)) \cong \sN \]
as $\sU$-modules by \cite[Theorem 9.5]{DCapOne}. It follows that  $\sM \cong \sN \cong \widetilde{M}$ as $\O$-modules, so $\sM$ is $\O$-coherent as required.
\end{proof}
\subsection{Construction of simple \ts{\w{\sU(\sL_X)}}-modules}\label{ConstrSimple}
We finish this paper by presenting a representation-theoretic application of our Kashiwara equivalence. Given a $K$-affinoid algebra $A$, let $X = \Sp(A)$ and let $A_\fr{m} := \O_{X,x}$ be the stalk of the structure sheaf $\O_X$ at the point $x \in X$ defined by a maximal ideal $\fr{m}$ of $A$. Thus $A_{\fr{m}}$ is the direct limit the affinoid algebras $\O(Y)$ running over all affinoid subdomains $Y$ of $X$ containing $x$.

\begin{lem} Let $A$ be a $K$-affinoid algebra and let $\mathfrak{m}$ be a maximal ideal of $A$ such that $A_\mathfrak{m}$ is a regular local ring. Let $L$ be a $(K,A)$-Lie algebra which admits an $\mathfrak{m}$-standard basis. Then $L_\mathfrak{m}=A_\mathfrak{m}\otimes_A L$ is naturally a $(K,A_\mathfrak{m})$-Lie algebra and the $U(L_{\mathfrak{m}})$-module $A_\mathfrak{m}$ is simple.
\end{lem}
\begin{proof} We first note that for each affinoid subdomain $Y$ of $\Sp(A)$, $\O(Y)\otimes_A L$ has a canonical $(K,\O(Y))$-Lie algebra structure determined by the $(K,A)$-Lie algebra structure on $L$ by \cite[Corollary 2.4]{DCapOne}. By considering the affinoid neighbourhoods of $x$ and using \cite[Lemma 2.2]{DCapOne} we can thus deduce that $L_\mathfrak{m}$ is a $(K,A_\mathfrak{m})$-Lie algebra. 

Let $\{x_1,\ldots,x_d\}$ be the $\mathfrak{m}$-standard basis for $L$ and let $f_1,\ldots, f_r$ be the corresponding generating set for $\mathfrak{m}$. Since $A_\mathfrak{m}$ is a regular local ring by assumption, the associated graded ring $\gr_{\mathfrak{m}} A$ of $A$ with respect to the $\mathfrak{m}$-adic filtration is isomorphic to the polynomial ring $F[y_1,\ldots,y_r]$ where $F = A / \mathfrak{m}$ and $y_i = f_i + \mathfrak{m}^2$. The derivations $ \rho(x_i)$ of $A$ send $\mathfrak{m}^n$ to $\mathfrak{m}^{n-1}$ for all $n \geq 1$ and induce the $F$-linear derivations $\partial / \partial y_i$ on $\gr_\mathfrak{m} A$. Since $A / \mathfrak{m}$ is a field of characteristic zero, it is well-known that $F[y_1,\ldots, y_r]$ has no non-trivial ideals stable under all these derivations. Now if $J$ is a $U(L_{\mathfrak{m}})$-submodule of $A_\mathfrak{m}$ then $\gr_\mathfrak{m} J$ is an ideal of $\gr (A_\mathfrak{m}) \cong \gr_\mathfrak{m} A$ stable under all $\partial/\partial y_i$. So $\gr_\mathfrak{m}J$ is either zero or all of $\gr A_\mathfrak{m}$ and the result follows because the $\mathfrak{m}_{\mathfrak{m}}$-adic filtration on $A_\mathfrak{m}$ is Zariskian.
\end{proof}

\begin{prop} Let $Y$ be a smooth, connected rigid analytic variety and let $\sL_Y$ be a Lie algebroid on $Y$ with surjective anchor map $\rho_Y : \sL_Y \to \T_Y$. Then $\O_Y$ is a simple coadmissible $\w{\sU(\sL_Y)}$-module.
\end{prop}
\begin{proof} We suppose that $Y$ is not a single point as the statement is trivially true in this case. Since $\O_Y$ is an $\O_Y$-coherent $\sL_Y$-module, it is a coadmissible $\w{\sU(\sL_Y)}$-module by Theorem \ref{OxCoh}. Let $\J$ be a coadmissible $\w{\sU(\sL_Y)}$-submodule of $\O_Y$. Then $\J$ is a coherent ideal of $\O_Y$ by Corollary \ref{OxCoh}, and $\J$ is stable under $\rho(\sL_Y) = \T_Y$. We will show that $\Supp(\J) \cap \Supp(\O_Y/\J)$ is empty. Because $\J$ is coherent, these supports are closed analytic subspaces of $Y$, so the connectedness of $Y$ will then imply that one of them is empty and thus either $\J = 0$ or $\J = \O_Y$.

Suppose for a contradiction that $y \in  \Supp(\J) \cap \Supp(\O_Y/\J)$. Choose a connected affinoid subdomain $U$ of $Y$ containing $y$ and let $\mathfrak{m}$ be the ideal of functions in $\O(U)$ vanishing at $y$. If $\mathfrak{m} = \mathfrak{m}^2$ then $\mathfrak{m} = 0$ by Lemma \ref{BasisB} so $U = \{y\}$. The connectedness of $Y$ then forces $Y = \{y\}$ which we assumed not to be the case at the outset. So $\mathfrak{m} \neq \mathfrak{m}^2$, and after shrinking $U$ and applying Proposition \ref{StdBasis} we may assume that $\sL_Y(U)$ has an $\mathfrak{m}$-standard basis. 

Now because $y \in \Supp(\O_Y/J) \cap \Supp(\J)$ and because $Y$ is smooth, the stalk $\J_y$ of $\J$ at $y$ is a proper, non-zero ideal of the regular local ring $\O_{Y,y} = \O(U)_{\mathfrak{m}}$. Because it is $\sL_{Y,y} = \sL_Y(U)_\fr{m}$-stable by construction we obtain a contradiction after applying the Lemma.
\end{proof}

\begin{thm} Let $X$ be a smooth rigid analytic variety and let $\sL_X$ be a Lie algebroid on $X$ with surjective anchor map $\rho_X : \sL_X \to \T_X$.
\be
\item $\iota_+ \O_Y$ is a simple coadmissible $\w{\sU(\sL_X)}$-module whenever $\iota : Y \to X$ is the inclusion of a smooth, connected, closed subvariety $Y$. 
\item If $\iota' : Y\to X$ is another such inclusion and $\iota_+ \O_Y \cong \iota'_+ \O_{Y'}$ as coadmissible $\w{\sU(\sL_X)}$-modules, then $Y = Y'$.
\ee \end{thm}
\begin{proof} (a) Since $\rho_X : \sL_X \to \T_X$ is surjective, so is $\iota^\ast \rho_X : \iota^\ast \sL_X \to \iota^\ast \T_X$. Because $Y$ is smooth, the normal sheaf $\N_{Y/X}$ is locally free and $\theta : \iota^\ast \T_X \to \N_{Y/X}$ is surjective by \cite[Proposition 2.5]{BLR3}. Thus the hypotheses of Theorem \ref{KashGenLeft} are satisfied, so by the equivalence of categories the coadmissible $\w{\sU(\sL_X)}$-submodules of $\iota_+ \O_Y$ are in bijective correspondence with the coadmissible $\w{\sU(\sL_Y)}$-submodules of $\O_Y$. But $\rho_Y : \sL_Y \to \T_Y$ is surjective because $\iota^\ast \rho_X : \iota^\ast \sL_X \to \iota^\ast \T_X$ is surjective, so that $\O_Y$ is a simple coadmissible $\w{\sU(\sL_Y)}$-module by the Proposition. 

(b) It will be enough to show that $\Supp (\iota_+ \O_Y) = Y$. We recall the basis $\B$ from the proof of Theorem \ref{KashiwaraGeneral}. Now if $\sM$ is a coadmissible $\w{\sU(\sL_X)}$-module then for any $U \in \B$, $\sM(U) = 0$ if and only if $\sM|_U = 0$. Hence $\Supp \sM$ is the complement of the union of all $U \in \B$ such that $\sM(U) = 0$. The definition of $\iota_+\O_Y$, Corollary \ref{Lieflat} and \cite[Proposition 7.5(c)]{DCapOne} now show if $U \in \B$ then $(\iota_+\O_Y)(U) = 0$ if and only if $U \cap Y = \emptyset$. Thus
\[\Supp (\iota_+ \O_Y) = X \backslash \bigcup \{ U \in \B : U \cap Y = \emptyset \} = Y\]
because $X \backslash Y$ is an admissible open subset of $X$. 
\end{proof}

\appendix
\section{Pullback and pushforward of abelian sheaves}
\subsection{Closed embeddings}\label{ShSuppClsdSub}

Suppose that $\iota\colon Y\to X$ is a closed embedding of rigid $K$-analytic spaces defined by the radical, coherent $\O_X$-ideal $\I$. We will frequently identify $Y$ with its image in $X$. By \cite[Definition 9.3.1/4]{BGR} and \cite[\href{http://stacks.math.columbia.edu/tag/00X6}{Proposition 00X6}]{stacks-project}, there is a morphism of sites $\iota : Y_{\rig} \to X_{\rig}$ given by the continuous functor which sends an admissible open subset $U$ of $X$ to $\iota^{-1}(U) = U \cap Y$. It induces a pair of adjoint functors between categories of abelian sheaves
\[ \iota^{-1} : \Ab(X_{\rig}) \to \Ab(Y_{\rig}) \qmb{and} \iota_\ast : \Ab(Y_{\rig}) \to \Ab(X_{\rig})\]
by \cite[\href{http://stacks.math.columbia.edu/tag/00WX}{Lemma 00WX}]{stacks-project}. We recall the explicit definitions of these functors, following the discussion in \cite[\href{http://stacks.math.columbia.edu/tag/00VC}{Section 00VC}]{stacks-project}. Let $\F \in \Ab(Y_{\rig})$ and $\G \in \Ab(X_{\rig})$, and let admissible open $V\subset Y$ and $U \subset X$ be given. Then
\[ (\iota^{-1}\G)(V) = \dirlim \G(W) \qmb{and} \iota_\ast\F(U) = \F(U \cap Y) \]
where the direct limit is taken over all admissible open $W \subset X$ such that $V \subset W \cap Y$. 

We define the \emph{support} of the abelian sheaf $\G$ on $X_{\rig}$ as follows:
\[ \Supp \G := X \backslash \bigcup \{ U \subseteq X : U \mbox{ is an admissible open and } \G|_U = 0\}\]
and we say that \emph{$\G$ is supported on $Y$} if $\Supp \G \subseteq Y$, or equivalently, if $\G|_U= 0$ for every admissible open $U \subseteq X$ such that $U \cap Y = \emptyset$. 
\begin{thm} Let $\iota : Y \to X$ be a closed embedding of rigid $K$-analytic spaces. Then the functor $\iota_\ast$ induces an equivalence of categories between $\Ab(Y_{\rig})$ and the full subcategory $\Ab(X^Y_{\rig})$ consisting of sheaves supported on $Y$.
\end{thm}

The corresponding result for ordinary topological spaces is completely standard, and Theorem \ref{ShSuppClsdSub} is presumably well-known, but we were unable to locate a complete proof in the literature. Since our proof involves some non-trivial ideas, we have decided to give the details here for the convenience of the reader. 

\subsection{Two useful results on affinoid varieties}\label{TubNhd}

Suppose now that $X$ is affinoid, so that $Y$ is a closed analytic subset defined by the ideal $\I(X)$ of $\O(X)$. Let $f_1,\ldots,f_r$ generate $\I(X)$ as an ideal. For every $n\geq 0$, we call 
\[ Y_n := X\left( \frac{f_1}{\pi^n}, \ldots, \frac{f_r}{\pi^n} \right)\]
of $X$ a \emph{tubular neighbourhood} of $Y$. Clearly $Y_n$ is an affinoid subdomain of $X$ containing $Y$.

\begin{prop} Let $Y$ be a closed analytic subset of the $K$-affinoid variety $X$. Every admissible open subset $U$ of $X$ containing $Y$ also contains $Y_n$ for some $n \geq 0$.
\end{prop}
\begin{proof} See \cite[p.52]{ConradSurvey} or \cite[Lemma 2.3]{Kisin99}.\end{proof}

The other ingredient in our proof of Theorem \ref{ShSuppClsdSub} is the following

\begin{lem} Let $Y$ be a closed analytic subset of the $K$-affinoid variety $X$, and let $V$ be a rational subdomain of $Y$. Then there is a rational subdomain $U$ of $X$ such that $V = U \cap Y$.
\end{lem}
\begin{proof} Note that $Y$ is itself affinoid by \cite[Proposition 9.4.4/1]{BGR}, and that $\O(Y) \cong \O(X) / \I(X)$.  By \cite[Proposition 7.2.4/1]{BGR}, there is a factorisation \[V= V_{m}\to V_{m-1}\to \cdots\to V_1 = Y\] such that $V_{k+1}=V_k(g_k)$ or $V_{k+1}=V_k(1/g_k)$ for some $g_k\in \O(V_k)$. By induction on $m$ we can therefore assume that $V = Y(g)$ or $V = Y(1/g)$ for some $g \in Y$. Suppose first that $V = Y(g)$. Since the map $\O(X) \to \O(Y)$ is surjective, we can find some preimage $h \in \O(X)$ of $g \in \O(Y)$ and define $U := X(h)$. Now applying the $\Sp$ functor to the natural isomorphism
\[  \frac{\O(X)\langle t\rangle}{\langle t - h\rangle} \underset{\O(X)}{\otimes} \frac{\O(X)}{\I(X)} \stackrel{\cong}{\longrightarrow} \frac{ \O(Y) \langle t \rangle }{ \langle t - g \rangle } \]
of $K$-Banach algebras shows that 
\[ U \cap Y = X(h) \cap Y = Y(g) =  V\]
as required. The case where $V = Y(1/g)$ is entirely similar.
\end{proof}

\subsection{Proof of Theorem \ref{ShSuppClsdSub}}
Let $\F$ be an abelian sheaf on $Y$. Then $\iota_\ast \F$ is supported on $Y$. We will show that the counit morphism 
\[\epsilon_\F : \iota^{-1} \iota_\ast \F \to \F\]
is an isomorphism. By \cite[Lemma A.1]{DCapOne} we can assume that $X$ is affinoid. Let $Y'\subset Y$ be an admissible open subset, and choose an admissible covering of $Y'$ consisting of rational subdomains in $Y$. By appealing to \cite[Lemma A.1]{DCapOne} again, it is enough to show that 
\[ (\iota^{-1}\iota_\ast\F)(V) \to \F(V)\]
is an isomorphism for every rational subdomain $V$ of $Y$. Now $(\iota^{-1}\iota_\ast\F)(V)$ is the direct limit of the $\F(W \cap Y)$ where $W$ ranges over all open subdomains of $X$ such that $V \subset W \cap Y$. By Lemma \ref{TubNhd}, we can find a rational subdomain $U$ of $X$ such that $V = U \cap Y$, and the result follows.

Now let $\G$ be an abelian sheaf on $X$ which is supported on $Y$; we will show that the unit morphism $\eta_\G : \G \to \iota_\ast \iota^{-1} \G$ is an isomorphism. By \cite[Lemma A.1]{DCapOne}, it is enough to show that $\eta_\G(X)$ is an isomorphism whenever $X$ is affinoid. Let $f_1,\ldots,f_r$ generate $\I(X)$. By Proposition \ref{TubNhd}, there is an isomorphism
\[ \dirlim \G(Y_n) \stackrel{\cong}{\longrightarrow} (\iota^{-1}\G)(Y) = \iota_\ast \iota^{-1}\G(X) \]
so it will be sufficient to show that the restriction morphism $\G(X) \to \G(Y_n)$ is an isomorphism for any $n\geq 0$. Let $f_{r+1} = \pi^n$. Since the elements $f_1,\ldots,f_r,f_{r+1} \in \O(X)$ have no common zero on $X$, we may consider the rational covering $\U$ of $X$ generated by these elements in the sense of \cite[\S 8.2.2]{BGR}:
\[ \U = \{U_1,\ldots,U_{r+1}\} \qmb{where} U_i = X\left(\frac{f_1}{f_i},\ldots, \frac{f_{r+1}}{f_i}\right).\]
Thus $U_{r+1} = Y_n$ and $U_i \cap Y = \emptyset$ for all $i=1,\ldots,r$. Since $\G$ is supported on $Y$ by assumption, we see that $\G(U_i) = 0$ for all $i \leq r$, and now the sheaf condition satisfied by $\G$ shows that the restriction map $\G(X) \to \G(U_{r+1}) = \G(Y_n)$ is an isomorphism. \qed
\bibliography{references}
\bibliographystyle{plain}
\end{document}